\documentclass[12pt]{amsart}
\usepackage{amsmath}
\usepackage{amsthm}
\usepackage{amsfonts}
\usepackage{comment}
\usepackage{amssymb,amsrefs}
\usepackage{comment}
\usepackage{hyperref}
\usepackage{mathtools}
\usepackage{bm}
\usepackage[margin=1in]{geometry}

\usepackage{mathrsfs}
\usepackage{enumerate}

\usepackage{todonotes}

\newtheorem{theorem}{Theorem}[section]

\newtheorem{lemma}[theorem]{Lemma}

\newtheorem{conjecture}[theorem]{Conjecture}

\theoremstyle{definition}

\newtheorem*{theorem*}{Theorem}
\newtheorem*{proposition*}{Proposition}
\newtheorem*{lemma*}{Lemma}

\theoremstyle{remark}
\newtheorem*{remark}{Remark}

\numberwithin{equation}{section}

\newcommand{\CC}{\mathbb C}
\newcommand{\FF}{\mathbb F}
\newcommand{\QQ}{\mathbb Q}
\newcommand{\OO}{\mathcal O}

\newcommand{\RR}{\mathbb R}

\newcommand{\on}{\operatorname}

\makeatletter
\newcommand*{\defeq}{\mathrel{\rlap{%
                     \raisebox{0.3ex}{$\m@th\cdot$}}%
                     \raisebox{-0.3ex}{$\m@th\cdot$}}%
                     =}
\makeatother

\newcommand{\pp}{\mathfrak p}

\newcommand\half{\frac12}

\renewcommand{\Re}{\operatorname{Re}}

\DeclareMathOperator{\Sym}{Sym}
\newcommand{\Lij}{\Lambda_{i \otimes j}}

\newcommand{\Lt}{\widetilde{\Lambda}}
\newcommand{\Nij}{N_{i \otimes j}}
\newcommand{\PSY}[2]{\Sym^{#1} E_1 \otimes \Sym^{#2} E_2}

\newcommand{\zzt}{\left\lvert \frac{\zeta'}{\zeta}(2) \right\rvert}
\newcommand{\gij}{\gamma_{i \otimes j}}
\newcommand{\dgij}{\frac{\gij'}{\gij}}
\newcommand{\dgk}{\frac{\gamma_{k}'}{\gamma_{k}}}
\newcommand{\dGam}{\frac{\Gamma'}{\Gamma}}
\newcommand{\ddG}{\left( \frac{\Gamma'}{\Gamma} \right)'}
\newcommand{\dLL}{\frac{L'_{i\otimes j}}{L_{i\otimes j}}}
\newcommand{\nc}{\newcommand}
\nc{\SWAG}{W}
\nc{\zt}{Z^{\on{triv}}}
\nc{\zn}{Z^{\on{non}}}

\allowdisplaybreaks[1]

\author[Evan Chen]{Evan Chen}\address
{ Department of Mathematics, Massachusetts Institute of Technology, 77 Massachusetts Avenue, Cambridge, MA 02139}
\email{\href{mailto:evanchen@mit.edu}{{\tt evanchen@mit.edu}}}

\author[Peter S. Park]{Peter S. Park}\address{ Department of Mathematics, Princeton University, Fine Hall, Washington Road, Princeton, NJ 08544}
\email{\href{mailto:pspark@math.princeton.edu}{{\tt pspark@math.princeton.edu}}}

\author[Ashvin A. Swaminathan]{Ashvin A. Swaminathan}\address{ Department of Mathematics, Harvard College, \\
1 Oxford Street, Cambridge, MA 02138}
\email{\href{mailto:aaswaminathan@college.harvard.edu}{{\tt aaswaminathan@college.harvard.edu}}}

\begin{document}
 
\title[Elliptic Curve Variants of Least Quadratic Nonresidue and Linnik's Theorem]{Elliptic Curve Variants of the Least Quadratic Nonresidue Problem and Linnik's Theorem}
\date{\today}

\begin{abstract}
Let $E_1$ and $E_2$ be $\overline{\mathbb{Q}}$-nonisogenous, semistable elliptic curves over $\mathbb{Q}$,
having respective conductors $N_{E_1}$ and $N_{E_2}$ and both without complex multiplication.
For each prime $p$, denote by $a_{E_i}(p) := p+1-\#E_i(\mathbb{F}_p)$ the trace of Frobenius.
Under the assumption of the Generalized Riemann Hypothesis (GRH) for the convolved symmetric power $L$-functions
$L(s, \mathrm{Sym}^i E_1\otimes\mathrm{Sym}^j E_2)$ where $i,j\in\{0,1,2\}$,
we prove an explicit result that can be stated succinctly as follows: there exists a prime $p\nmid N_{E_1}N_{E_2}$ such that $a_{E_1}(p)a_{E_2}(p)<0$ and
	\[ p < \big( (32+o(1))\cdot \log N_{E_1} N_{E_2}\big)^2. \]
This improves and makes explicit a result of Bucur and Kedlaya.

Now, if $I\subset[-1,1]$ is a subinterval with Sato-Tate measure $\mu$
and if the symmetric power $L$-functions $L(s, \mathrm{Sym}^k E_1)$ are functorial and satisfy GRH for all $k \le 8/\mu$, we employ similar techniques to prove an explicit result that can be stated succinctly as follows: there exists a prime $p\nmid N_{E_1}$ such that $a_{E_1}(p)/(2\sqrt{p})\in I$ and
	\[ p < \left((21+o(1)) \cdot \mu^{-2}\log (N_{E_1}/\mu)\right)^2. \]
\end{abstract}
\maketitle
\vspace*{-0.2cm}


\section{Introduction}\label{intro}


Let $E$ be an elliptic curve over $\QQ$, and for each prime $p$, let $\#E(\FF_{p})$ denote the number of rational points of $E$ over the finite field $\FF_{p}$ of $p$ elements. Let $a_E(p) =p + 1 - \# E(\FF_p)$ denote the trace of Frobenius, which, by the Hasse bound, satisfies the following inequality:
\begin{equation*}
	\left|a_E(p)\right| < 2 \sqrt {p}.
\end{equation*}
From the above bound, one readily observes that for each prime $p$, there exists a unique angle $\theta_{p} \in [0, \pi]$ such that $a_E(p) = 2 \sqrt {p} \cos \theta_{p}$. It is natural to ask how the angles $\theta_p$ are distributed for a fixed elliptic curve as $p$ ranges over the primes. The Sato-Tate conjecture, which was recently proven by Barnet-Lamb, Harris, Geraghty, and Taylor in~\cite{tater}, asserts that when $E$ does not have complex multiplication (CM), the angles $\theta_p$ are equidistributed with respect to the Sato-Tate measure $\mu_{\on{ST}}$ defined by $d\mu_{\on{ST}} \defeq \frac{2}{\pi} \sin^2 \theta \; d\theta$. The conjecture is stated formally as follows:
\begin{theorem*}
	[Sato-Tate Conjecture]
	Fix a non-CM elliptic curve $E/\QQ$, and let $I \subset [0,\pi]$ be a subinterval. Then, we have that
    \begin{equation*}
        \lim_{x \to \infty} \frac{\#\{p\leq x:\theta_p\in I\}}{\#\{p \leq x\}} 
        =         	\mu_{\on{ST}}(I) \defeq \displaystyle\int_I\frac{2}{\pi}\sin^2\theta~d\theta.
    \end{equation*}
\end{theorem*}
As a consequence of the Sato-Tate Conjecture, it is natural to consider elliptic curve variants of well-known questions on the distribution of primes; in this paper, we address two such problems. First, in analogy with the least quadratic nonresidue problem, we obtain an explicit, conditional bound on the least prime at which two non-CM elliptic curves have Frobenius traces of opposite signs. Then, in analogy with Linnik's Theorem, we obtain an explicit, conditional bound on the least prime $p$ for which the angle $\theta_p$ associated to a single elliptic curve lies in a specified interval. The rest of this section is devoted to introducing these two questions and providing detailed statements of our results for each question.

\subsection{Frobenius Traces of Opposite Sign}
We begin by recalling that two elliptic curves $E_1$ and $E_2$ over $\QQ$ are $\QQ$-isogenous if and only if their traces of Frobenius satisfy $a_{E_1}(p) = a_{E_2}(p)$ for all but finitely many primes $p$. In light of this fact, it is natural to ask the following question: given two nonisogenous elliptic curves, can one obtain a tight upper bound on the least prime $p$ for which $a_{E_1}(p) \neq a_{E_2}(p)$?\footnote{See the paper~\cite{smores} and the book~\cite{intro}*{Chapter 11.3} by Cojocaru and Murty for more on such Linnik-type problems for elliptic curves.} This question was first addressed by Serre (see~\cite{serre}); assuming the Generalized Riemann Hypothesis (GRH) for certain Artin $L$-functions, he used the $\ell$-adic properties of elliptic curves to show that the least such prime is $\ll (\log D)^2$, where $D$ denotes the maximum of the conductors of $E_1$ and $E_2$. Subsequently, Duke and Kowalski proved the following unconditional analogue of Serre's result:
\begin{theorem*}[Duke and Kowalski,~\cite{duke}]
For $D >0$, $\alpha >1$, let $P(D,\alpha)$ denote the maximal number of isogeny classes of elliptic curves $E/\QQ$ such that $E$ is non-CM and has squarefree conductor $\leq D$, and such that for all primes $p \leq (\log D)^\alpha$, they have the same trace of Frobenius at $p$. Then for any $\varepsilon > 0$, we have that $P(D, \alpha) \ll D^{\varepsilon+10/\alpha}$.
\end{theorem*}
     
Now, consider the following related question: what is the least prime $p$ for which $a_{E_1}(p)$ and $a_{E_2}(p)$ are not only different, but also satisfy $a_{E_1}(p) a_{E_2}(p) < 0$ (i.e.\ have opposite signs)? While studying the $\ell$-adic behavior of elliptic curves (in the manner of Serre~\cite{serre}) seems insufficient for solving this particular problem, we can instead exploit the distribution of angles $\theta_p$ given by the Sato-Tate conjecture to handle the opposite-sign condition. Using this alternative strategy, Bucur and Kedlaya show in~\cite{kedlaya} that if $L(s, \PSY ij)$ has an analytic continuation and functional equation,\footnote{These properties would be a consequence of the Langlands program.} and also satisfies GRH for all $i,j \geq 0$, then the smallest prime $p$ for which $a_{E_1}(p)a_{E_2}(p)<0$ satisfies $p \ll (\log N)^2 (\log \log 2N)^2$, where $N$ denotes the product of the conductors of $E_1$ and $E_2$.
In the first main theorem of this paper, we provide an explicit bound of the form $p \ll (\log N)^2$ on the least prime $p$ for which $a_{E_1}(p)a_{E_2}(p) <0$, under milder conditions than those assumed by Bucur and Kedlaya in~\cite{kedlaya}; indeed, our method only requires GRH for the $L$-functions $L(s, \PSY ij)$ with $i, j \in \{0,1,2\}$. Our result can be stated succinctly as follows; we refer the reader to~\eqref{eq:explicit_main} for the completely explicit version of the theorem.

\begin{theorem}\label{thm:main}
	Let $E_1$ and $E_2$ be $\overline{\QQ}$-nonisogenous, non-CM, semistable elliptic curves over $\QQ$, and let $N$ be the product of their conductors. Assume that GRH holds for the $L$-functions $L(s, \PSY ij)$ with $i, j \in \{0,1,2\}$. Then, there exists a prime $p$ such that $a_{E_1}(p)a_{E_2}(p) < 0$ and
	\[ p < \left( 32\log N + O\big(\sqrt[3]{\log N}\big)\right)^2. \]
\end{theorem}

\begin{remark}
One key step in our proof of Theorem~\ref{thm:main} is estimating the digamma function (i.e.\ the logarithmic derivative of the $\Gamma$-function). The estimate we use (see Lemmas~\ref{lem:bitpar} and~\ref{lem:digamma}) depends on the conjectured values of the local parameters at infinity for the symmetric power $L$-functions $L(s, \Sym^k E)$. Given bounds on the local parameters at infinity for elliptic curves over a totally real number field $K/\QQ$, an analogue of Theorem~\ref{thm:main} can be obtained using our methods, yielding a bound of the form $O\big([K : \QQ]^2(\log N)^2\big)$.
\end{remark}

\subsection{Sato-Tate Angles in a Given Interval}
Given an elliptic curve $E/\QQ$ with conductor $N_E$, it is natural to ask the following question: what is the least prime $p$ such that the associated Sato-Tate angle $\theta_p$ is contained in a fixed subinterval $[\alpha, \beta] \subset [0, \pi]$? For the case of CM elliptic curves over $\QQ$, this question is answered in~\cite{us}, where it is shown unconditionally that the least such prime $p$ is bounded above by a polynomial in $\frac{N_E}{\beta-\alpha}$. The case of non-CM elliptic curves is addressed by Lemke Oliver and Thorner in~\cite{thorner}, where a bound (without GRH) for the least such prime $p$ is obtained in terms of the number of symmetric power $L$-functions $L(s, \Sym^k E)$ that are known to have analytic continuations and functional equations of the usual kind. In this paper, we obtain an explicit analogue of Lemke Oliver and Thorner's result under the additional assumption of GRH for the symmetric-power $L$-functions $L(s, \Sym^k E)$. Our complete set of hypotheses is listed in the following conjecture, after which we state our second main theorem:

\begin{conjecture}\label{conj:scrub}
Let $E/\QQ$ be a non-CM, semistable elliptic curve with conductor $N_E$.  For each $k  \geq 0$, the following statements are true:

\begin{enumerate}
\item The conductor of $L(s, \Sym^k E)$ is $N_E^k$.
\item The $\gamma$-factor of $L(s, \Sym^k E)$ is given by
\[
\gamma(s, \Sym^k E) =
	\begin{cases}
		\displaystyle \prod_{j=1}^{(k+1)/2} \Gamma_\CC(s+j-1/2) & \text{ if $k$ is odd,} \\
		\Gamma_\RR(s+ r)\displaystyle \prod_{j=1}^{k/2} \Gamma_\CC(s+j) & \text{ if $k$ is even,}
	\end{cases}
\]
where $\Gamma_\RR(s) \defeq \pi^{-s/2}\Gamma(s/2)$, $\Gamma_\CC(s) \defeq 2(2\pi)^{-s}\Gamma(s)$, and $r \in \{0, 1\}$ is such that $r \equiv k/2 \pmod 2$.

\item In the Euler product that defines $L(s, \Sym^k E)$, namely
\[
L(s, \Sym^k E) = \prod_{p \nmid N_E}\prod_{j=0}^k \left( 1 - e^{i\theta_p(k-2j)} p^{-s} \right)^{-1}
	\cdot \prod_{p \mid N_E} L_p(s, \Sym^k E),
\]
we have for each prime $p \mid N_E$  that $L_p(s, \Sym^k E) = (1 - (-\lambda_p / \sqrt{p})^k p^{-s})^{-1}$, where $\lambda_p \in \{-1, 1\}$ denotes the eigenvalue of the Atkin-Lehner operator $W(p)$ acting on $E$.

\item The completed $L$-function $\Lambda(s, \Sym^k E) \defeq (N_E^k)^{s/2} \gamma(s, \Sym^k E) L(s, \Sym^k E)$ satisfies the functional equation
\[
\Lambda(s, \Sym^k E) = w_{\Sym^k E}\Lambda(1-s, \Sym^k E)
\]
where $w_{\Sym^k E}$ is a complex number of absolute value 1. Furthermore, the function $\left(\frac{1}{2}s(1-s)\right)^{e_k}\Lambda(s, \Sym^k E)$ is entire, where $e_k = 1$ if $k=0$ and $e_k = 0$ otherwise.

\item The Generalized Riemann Hypothesis (GRH); i.e.\ each zero of $\Lambda(s, \Sym^k E)$ has real part equal to $1/2$.

\end{enumerate}
\end{conjecture}
\begin{remark} 
Statements (1)--(4) in Conjecture~\ref{conj:scrub} are reasonable assumptions to make, because Langlands functoriality would imply that $L(s, \Sym^k E)$ is the $L$-function of a cuspidal automorphic representation on $\operatorname{GL}_{k+1}(\mathbb{A}_\QQ)$. In~\cite{cogdell}, Cogdell and Michel computed the $L$-function data stated in Conjecture~\ref{conj:scrub}, under the assumption of a global lifting map on automorphic representations that is compatible with the local Langlands correspondence. The existence of such a lifting map has been shown for all $k \le 4$, thus proving statements (1)--(4) for these cases (see ~\cites{sym2,sym3,sym4,sym12}). Finally, note that statements (1)--(4) can be generalized to any non-CM newform with squarefree level and positive, even weight, whereas GRH is unknown for all $L$-functions.
\end{remark}
With the hypotheses stated in Conjecture~\ref{conj:scrub}, we obtain the following Linnik-type result; once again, we refer the reader to~\eqref{eq:explicit_roar} for the completely explicit version of the theorem.
\begin{theorem}\label{thm:roar}
Let $E/\QQ$ be a non-CM, semistable elliptic curve with conductor $N_E$.
Let $I \subseteq [0,\pi]$ be an interval with Sato-Tate measure $\mu$.
Under the assumption of Conjecture~\ref{conj:scrub} for $L(s, \Sym^k E)$
for all $k \leq 8/\mu$,
there exists a prime $p$ such that $\theta_p \in I$ and
\[ p < \left(\frac{21\log(N_E/\mu) + O\big(\sqrt[3]{\log(N_E/\mu)}\big)}{\mu^2}\right)^2. \]
\end{theorem}

The rest of this paper is organized as follows. Section~\ref{defs} both presents an introduction to the analytic theory of symmetric power $L$-functions associated to elliptic curves and discusses the fundamental tools that we employ in our proofs of Theorems~\ref{thm:main} and~\ref{thm:roar}. Next, Section~\ref{proof1} details the main ideas behind the proof of Theorem~\ref{thm:main}, and Section~\ref{sec:pain} completes the proof of Theorem~\ref{thm:main}, making all estimates in Section~\ref{proof1} fully explicit. In Section~\ref{likith}, we provide a proof of Theorem~\ref{thm:roar}. Finally, in Appendix~\ref{rawrconstants}, we discuss the optimality of our estimates, and in Appendix~\ref{tabwars}, we make a list of all constants and variables used throughout the paper.
\section{Symmetric Power $L$-Functions of an Elliptic Curve}\label{defs}
In this section, we provide a brief discussion of the symmetric power $L$-functions associated to elliptic curves, as well as the Rankin-Selberg convolutions of two such $L$-functions. Many of the definitions and properties stated in this section are employed in our proofs of Theorems~\ref{thm:main} and~\ref{thm:roar}.

\subsection{Properties of $L(s, \Sym^k E)$}\label{kittypurry} Fix a semistable elliptic curve $E$ over $\QQ$. It clearly follows from statement (3) in Conjecture~\ref{conj:scrub} and the Hasse bound that $L(s, \Sym^k E)$ satisfies the Ramanujan-Petersson Conjecture. In our proof of Theorem~\ref{thm:main}, we make use of the local parameters of $L(s, \Sym^k E)$ at infinity for $k \in \{0, 1, 2\}$, and in our proof of Theorem~\ref{thm:roar}, we make use of the local parameters of $L(s, \Sym^k E)$ at infinity for all $k \geq 0$. The values of these parameters can be derived from statement (2) of Conjecture~\ref{conj:scrub}, and for convenience, we state them explicitly as follows (with multiplicity):
\begin{enumerate}
\item If $k$ is odd, then the local parameters at infinity for $L(s, \Sym^k E)$ are given by \mbox{$(2j+1)/2$,} where $j \in \{0, \dots, (k+1)/2\}$.  These parameters all have multiplicity $2$ except for $1/2$ and $(k+2)/2$, which have multiplicity $1$.
\item If $k \equiv 2 \pmod 4$, then the local parameters at infinity for $L(s, \Sym^k E)$ are given by $j$, where $j \in \{1, \dots, k/2 + 1\}$. These parameters all have multiplicity $2$ except for $(k+2)/2$, which has multiplicity $1$.
\item If $k \equiv 0 \pmod 4$ and $k > 0$, then the local parameters at infinity for $L(s, \Sym^k E)$ are given by $j$, where $j \in \{0, \dots, k/2+1\}$. These parameters all have multiplicity $2$ except for $0$, $1$, and $(k+2)/2$, which have multiplicity $1$.
\end{enumerate}
Observe that we can express the logarithmic derivative of $L(s, \Sym^k E)$ as a Dirichlet series in the following way:
\begin{equation}\label{eq:psy1}
	-\frac{L'}{L}(s, \Sym^k E)
	=
	\sum_{n = 1}^\infty \Lambda_k(n) n^{-s},
\end{equation}
where the coefficients $\Lambda_k(n)$ are supported on prime powers and are defined by
\[
	\Lambda_k(n) = \begin{cases}
		(\log p) U_k(\cos m\theta_{p})
			& \text{if $n = p^m$ for $p \nmid N$, $m \geq 1$}\\
		(\log p)t_{k,p,m}p^{-km/2}	& \text{if $n = p^m$ for $p \mid N$, $m \geq 1$} \\
		0 & \text{else}.
\end{cases} \]
Here, $|t_{k,p,m}| = 1$ for all $k,p,m$, and $U_k(x)$ is the $k^\mathrm{th}$ Chebyshev polynomial of the second kind and is defined recursively as follows: let $U_0(x) = 1$ and $U_1(x) = 2x$, and for each integer $k \geq 2$, let $U_k(x) = 2xU_{k-1}(x) - U_{k-2}(x)$.

We now provide an estimate on the logarithmic derivative of the $\gamma$-factor $\gamma(s, \Sym^k E)$ associated to $L(s, \Sym^k E)$; for ease of notation, we shall write $\gamma_k(s) = \gamma(s, \Sym^k E)$, which is unambiguous because we have fixed our curve $E$. To do so, we will require the following lemma, which is due to Ono and Soundararajan (see the proof of Lemma 4 in~\cite{onobound}):
\begin{lemma}\label{arouse}
	If $\sigma \geq 1$, then we have
	\[ \left\lvert \frac{\Gamma'}{\Gamma}(s) \right\rvert \le \frac{11}{3} + \log(|s|+1). \] 
\end{lemma}
 We now state and prove the digamma estimate that we use in the proof of Theorem~\ref{thm:main}:
\begin{lemma} \label{lem:bitpar}
	Let $k \ge 0$. If $s\in \RR$ is such that $-1/4 < s < 0$, then we have that
    \[ 
    	\left\lvert \dgk (s) \right\rvert
    	\le \frac {k+1}{2} \left(\frac{11}{3} + \log \pi +
		\log\left(\frac{s}{2} + \frac{k+24}{8} \right) \right) + \frac{28}{5} + \frac{\Delta}{|s|},
    \]
    where we define $\Delta \defeq 1$ when $k \equiv 0 \pmod 4$ and $\Delta \defeq 0$ otherwise. If $s \in \RR$ is such that $s \geq 1$, then we have that
    \[\left\lvert \dgk (s) \right\rvert
    	\le \frac {k+1}{2} \left(\frac{11}{3} + \log \pi +
		\log\left(\frac{s}{2} + \frac{k+16}{8} \right) \right) + 1.  \]
    \end{lemma}
\begin{proof}
Statement (2) of Conjecture~\ref{conj:scrub} gives us an expression for the $\Gamma$-factor $\gamma_k(s)$. 
By the duplication formula of the $\Gamma$-function we have
\[
	\Gamma_{\mathbb C}(s) = \pi^{-s+\frac12} \Gamma\left( \frac s2 \right) \Gamma\left( \frac{s+1}{2} \right).
\]
We use this to rewrite our expression in the traditional form of the $\Gamma$-factor
(see~\cite[(5.3)]{iwaniec}).
Then taking the logarithmic derivative of the result yields that
$$\frac{\gamma_k'}{\gamma_k}(s) = \begin{cases}
		 \frac{k+1}{2}(-\log \pi) + \frac{1}{2}\left[\frac{\Gamma'}{\Gamma}\left(\frac{s+1/2}{2}\right) + \frac{\Gamma'}{\Gamma}\left(\frac{s+(k+2)/2}{2}\right)\right] + & \\ \qquad \sum_{j=1}^{(k-1)/2}\frac{\Gamma'}{\Gamma}\left(\frac{s+j+1/2}{2}\right) & \text{ if $k$ is odd,} \\
		 \frac{k+1}{2}(-\log \pi) +  \frac{1}{2}\frac{\Gamma'}{\Gamma}\left(\frac{s+k/2+1}{2}\right) +  \sum_{j=1}^{k/2} \frac{\Gamma'}{\Gamma}\left(\frac{s+j}{2}\right) & \text{ if $k \equiv 2 \text{ (mod 4)}$,} \\
        \frac{k+1}{2}(-\log \pi) + \frac{1}{2}\left[\frac{\Gamma'}{\Gamma}\left(\frac{s}{2}\right)+\frac{\Gamma'}{\Gamma}\left(\frac{s+1}{2}\right)+\frac{\Gamma'}{\Gamma}\left(\frac{s+k/2+1}{2}\right)\right] + & \\
        \qquad \sum_{j=1}^{(k-2)/2} \frac{\Gamma'}{\Gamma}\left(\frac{s+j+1}{2}\right) & \text{ if $k \equiv 0 \text{ (mod 4)}$,}
	\end{cases}$$
To estimate the digamma terms in the above expressions, we can use Lemma~\ref{arouse} on all terms except those with argument having real part less than $1$. For the remaining terms, we can apply the identity
\begin{equation}\label{eq:starshipsyo}
\frac{\Gamma'}{\Gamma}(s+1) = \frac{1}{s} + \frac{\Gamma'}{\Gamma}(s)
\end{equation}
sufficiently many times to translate their arguments so that they have real part at least $1$, and only then can we use Lemma~\ref{arouse} to estimate them. Finally, once we have applied Lemma~\ref{arouse}, which changes the digamma terms to logarithm terms, we can use the arithmetic-geometric-mean inequality to estimate the sums of logarithms. For example, consider the case where $k \equiv 2 \pmod 4$ and $s \geq 1$. Then Lemma~\ref{arouse} and the arithmetic-geometric-mean inequality combine to yield the bound
\begin{align*}
&\hphantom{\le} \left|\frac{1}{2}\frac{\Gamma'}{\Gamma}\left(\frac{s+k/2+1}{2}\right) +  \sum_{j=1}^{k/2} \frac{\Gamma'}{\Gamma}\left(\frac{s+j}{2}\right)\right| \\
&\leq 
\left\lvert \frac12\left( \frac{11}{3}+\log\left( \frac{s+k/2+1}{2} + 1 \right) \right)
+ \sum_{j=1}^{k/2} \left( \frac{11}{3} + \log\left( \frac{s+j}{2} + 1 \right) \right) \right\rvert \\
& \leq \frac{k+1}{2}\left(\frac{11}{3} + \log\left(\frac{s}{2} + \frac{(k/2+1)^2}{2(k+1)}+1\right)\right) \\
& \leq \frac{k+1}{2}\left(\frac{11}{3} + \log\left(\frac{s}{2} + \frac{k+12}{8}\right)\right). 
\end{align*}
Applying the general strategy described above, we obtain the following estimates:
\begin{enumerate}
\item When $-1/4 < s < 0$ and $k$ is odd, we have that
$$\left\lvert \dgk (s) \right\rvert
    	\le \frac {k+1}{2} \left(\frac{11}{3} + \log \pi +
		\log\left(\frac{s}{2} + \frac{k+21}{8}\right) \right) + \frac{28}{5}.$$
\item When $-1/4 < s < 0$ and $k \equiv 2 \pmod 4$, we have that
$$\left\lvert \dgk (s) \right\rvert
    	\le \frac {k+1}{2} \left(\frac{11}{3} + \log \pi +
		\log\left(\frac{s}{2} + \frac{k+24}{8}\right) \right) + \frac{80}{21}.$$
\item When $-1/4 < s < 0$ and $k \equiv 0 \pmod 4$, we have that
$$\left\lvert \dgk (s) \right\rvert
    	\le \frac {k+1}{2} \left(\frac{11}{3} + \log \pi +
		\log\left(\frac{s}{2} + \frac{k+24}{8}\right) \right) + \frac{64}{21} + \frac{1}{|s|}.$$
\item When $s \geq 1$ and $k$ is odd, we have that
$$\left\lvert \dgk (s) \right\rvert
    	\le \frac {k+1}{2} \left(\frac{11}{3} + \log \pi +
		\log\left(\frac{s}{2} + \frac{k+16}{8}\right) \right) + \frac{2}{3}.$$
\item When $s \geq 1$ and $k \equiv 2 \pmod 4$, we have that
$$\left\lvert \dgk (s) \right\rvert
    	\le \frac {k+1}{2} \left(\frac{11}{3} + \log \pi +
		\log\left(\frac{s}{2} + \frac{k+12}{8}\right) \right).$$
\item When $s \geq 1$ and $k \equiv 0 \pmod 4$, we have that
$$\left\lvert \dgk (s) \right\rvert
    	\le \frac {k+1}{2} \left(\frac{11}{3} + \log \pi +
		\log\left(\frac{s}{2} + \frac{k+16}{8}\right) \right) + 1.$$
\end{enumerate}
The lemma follows readily from the above six results.
\end{proof}

\subsection{Properties of $L(s, \PSY{i}{j})$}
Given two nonisogenous, non-CM, semistable elliptic curves $E_1$ and $E_2$ over $\QQ$ with conductors $N_{E_1}$ and $N_{E_2}$, respectively, we denote by $L(s, \PSY ij)$ the Rankin-Selberg convolution of $L(s, \Sym^i E_1)$ and $L(s, \Sym^j E_2)$, which is defined by the following Euler product:
$$L(s, \PSY{i}{j}) = \prod_p \left(1 - \alpha(p)\beta(p)p^{-s}\right)^{-1},$$
where $\alpha(n)$ and $\beta(n)$ respectively denote the Dirichlet series coefficients of $L(s, \Sym^i E_1)$ and $L(s, \Sym^j E_2)$.
Recall that by construction, $L(s, \PSY{i}{j})$ is self-dual and has degree $d_{ij} = (i+1)(j+1)$, and observe that we can express the logarithmic derivative of $L(s, \PSY{i}{j})$ as a Dirichlet series in the following way:
\begin{equation}\label{eq:psy}
	-\frac{L'}{L}(s, \PSY ij)
	=
	\sum_{n=1}^\infty \Lij(n) n^{-s},
\end{equation}
where the coefficients $\Lij(n)$ are supported on prime powers and are defined by
\[
	\Lij(n) = \begin{cases}
		(\log p) U_i(\cos m\theta_{1,p})U_j(\cos m\theta_{2,p})
			& \text{if $n = p^m$ for $p \nmid N$, $m \geq 1$}\\
		(\log p)t_{i,j,p,m}p^{-m(i+j)/2}	& \text{if $n = p^m$ for $p \mid N$, $m \geq 1$} \\
		0 & \text{else};
\end{cases} \]
again, we have that $|t_{i,j,p,m}| = 1$ for all $i,j,p,m$. We now make note of a number of analytic properties of $L(s, \PSY{i}{j})$:
\begin{enumerate}
\item The $L$-function $L(s, \PSY{i}{j})$ has a pole at $s = 1$ if and only if $i = j = 0$ (see the footnote in Section 1.7 of~\cite{lastsource}). In this case, we obtain the Riemann-$\zeta$ function, which has a pole of order $1$ and residue $1$ at $s = 1$. Later, it will be convenient to put $e_{k} = 1$ when $k = 0$ and $e_{k} = 0$ otherwise, as a means of quantifying the order of the possible pole at $s = 1$.
 \item For $0 \leq i,j \leq 4$ such that $(i,j) \neq (0,0)$, we have from the discussion in Section 1 of~\cite{brumley} that $L(s, \PSY{i}{j})$ has an analytic continuation to an entire function on $\CC$, as well as a functional equation.
\item For $0 \leq i,j \leq 4$, the arithmetic conductor $N_{i \otimes j}$ of $L(s, \PSY{i}{j})$ satisfies $N_{i \otimes j} \leq N_{E_1}^{j(i+1)} N_{E_2}^{i(j+1)}$.
\item For $0 \leq i,j \leq 4$, the local parameters at infinity of $L(s, \PSY{i}{j})$ are of the form $\kappa + \nu$, with $\kappa$ and $\nu$ being the local parameters at infinity of the $L$-functions $L(s, \Sym^{i} E_1)$ and $L(s, \Sym^{j} E_2)$, respectively, counted with multiplicity; see (7) of~\cite{brumley}.
\item We denote the $\gamma$-factor of $L(s, \PSY{i}{j})$ by $\gamma_{i \otimes j}(s)$. We present an important estimate for $\dgij(s)$ in Lemma~\ref{lem:digamma}.
\item Combining the information from properties (2)--(5) above, we obtain the following completed $L$-function:
\begin{equation}\label{eq:compsies}
\Lambda(s, \PSY{i}{j}) = (s(1-s))^{e_{i+j}}N_{i \otimes j}^{s/2}\gamma_{i \otimes j}(s)L(s, \PSY{i}{j}).
\end{equation}
which is known to be entire for $0 \leq i,j \leq 4$.
\end{enumerate}

The following lemma contains the digamma estimate that we will employ in our proof of Theorem~\ref{thm:main}:
\begin{lemma}
	\label{lem:digamma}
	Let $0 \leq i,j \leq 2$. If $s \in \RR$ is such that $-1/4 < s < 0$, then we have that
	\[ \left\lvert \dgij(s) \right\rvert
		\le  
		\frac{d_{ij}}{2} \left(\frac{11}{3} + \log\pi + \log\left(\frac{s}{2}+3\right) \right) + 5 + \frac{e_{i+j}}{|s|}.\]
If $s \in \RR$ is such that $s \geq 1$, then we have that
$$\left\lvert \dgij(s) \right\rvert
		\le  
		\frac{d_{ij}}{2} \left(  \frac{11}{3}+ \log\pi + \log\left(\frac{s}{2}+3\right)  \right)+ 1.$$
\end{lemma}

\begin{proof}
Recall the degree-$d_{ij}$ $\Gamma$-factor $\gamma_{i\otimes j}(s)$ is given by
\[ \dgij(s) = \prod_{\kappa, \nu} \pi^{-s/2} \Gamma\left( \frac{s+\kappa+\nu}{2} \right). \]
From our description of the local parameters at infinity for the symmetric power $L$-functions, we can use point (4) above to deduce that the local parameters at infinity for $L(s, \PSY{i}{j})$, where $0 \leq i,j \leq 2$, are nonnegative and bounded above by $4$. 

The lemma then follows readily by logarithmically differentiating
and imitating the argument used to prove Lemma~\ref{lem:bitpar}.
The term $5 + \frac{e_{i+j}}{|s|}$ in the first bound and the term $1$ in the second bound arises from having to shift some of the terms in the $\Gamma$-factor $\gamma_{i\otimes j}(s)$ so that their argument has real part at least $1$, in which case Lemma~\ref{arouse} can be applied. The factor of $d_{ij}$ results from applying Lemma~\ref{arouse} to each term in the degree-$d_{ij}$ $\Gamma$-factor $\gamma_{i\otimes j}(s)$.
The constant $3$ in the argument of the logarithm arises from dividing this bound of $4$ by $2$ and adding $1$. 
\end{proof}

In what follows, we put $\tau \defeq \frac{11}{3} + \log \pi$ for convenience, as this constant recurs often.
\section{Main Ideas for the Proof of Theorem~\ref{thm:main}}\label{proof1}

In this section, we provide a detailed discussion of the quantities that need to be estimated in order to prove Theorem~\ref{thm:main}. The method we employ is based in part on the work of Bach and Sorenson on computing explicit bounds for primes in residue classes (see~\cite{bach}).

\subsection{Initial Setup of the Proof}\label{coldasice}
Let $E_1$ and $E_2$ be two nonisogenous, non-CM, semistable elliptic curves over $\QQ$ with conductors $N_{E_1}$ and $N_{E_2}$, respectively, and let $N = N_{E_1} \cdot N_{E_2}$. Recall from Section~\ref{intro} that for each $i \in \{1, 2\}$ and every prime $p$, we can associate a Sato-Tate angle $\theta_{i,p} \in [0, \pi]$ to the curve $E_i$, such that the trace of Frobenius at $p$ is expressible as $a_{E_i}(p) = 2\sqrt{p}\cos\theta_{i,p}$.

Now, consider the polynomials $f_1(t)$ and $f_2(t)$, defined in terms of the Chebyshev polynomials $U_i(x)$ of the second kind (see Section~\ref{kittypurry} for the definition) by
\begin{align*}
	f_1(t) 
    &= U_0(t) + 2U_1(t) + U_2(t) = 4t(t+1),
    \text{ and} \\ 
	f_2(t) 
    &= f_1(- t) = U_0(t) - 2U_1(t) + U_2(t) = 4t(t-1).
\end{align*}
We introduce a sum $\SWAG$, which is defined in terms of the polynomials $f_i$ evaluated at $\cos \theta_{i,p}$:
\begin{equation}\label{eq:feds}
	\SWAG \defeq
	\sum_{\substack{p \leq x \\ p \nmid N}}\, (\log p) f_1(\cos \theta_{1,p})f_2(\cos \theta_{2,p}) \cdot \left( p/x\right)^a \log\left(x/p\right),
\end{equation}
where $a \in (0, 1/4]$ is a parameter whose value we will specify later.
Observe that we have $f_1(\cos\theta_{1,p})f_2(\cos\theta_{2,p}) > 0$ precisely when $\cos\theta_{1,p}\cos\theta_{2,p} <0$. It follows that if the sum $\SWAG$ is strictly positive, then there must be a prime $p \le x$ such that $a_{E_1}(p)a_{E_2}(p)<0$. Thus, to prove Theorem~\ref{thm:main}, it suffices to pick $x$ in such a way that $\SWAG > 0$.

In what follows, it will be convenient to express $\SWAG$ in terms of the von Mangoldt-type functions $\Lambda_{i\otimes j}(n)$. Consider the function $\Lt(n)$ defined as follows: if $n$ is not a prime power, then $\Lt(n) \defeq 0$, and if $n = p^m$, then $\Lt(n)$ is given by
\begin{align*}
	\Lt(n) &\defeq (\log p)f_1(\cos m\theta_{p,1})f_2(\cos m\theta_{p,2}) \\
	&= \Lambda_{0 \otimes 0}(n) + 2\Lambda_{1 \otimes 0}(n) - 2\Lambda_{0 \otimes 1}(n) + \Lambda_{2 \otimes 0}(n) + \Lambda_{0 \otimes 2}(n) - \\
    & \quad\,\, 4 \Lambda_{1 \otimes 1}(n) + 2\Lambda_{1 \otimes 2}(n)-2\Lambda_{2 \otimes 1}(n) + \Lambda_{2 \otimes 2}(n),
\end{align*}
where we recall from Section~\ref{defs} that $\Lij(n)$ is the $n^\mathrm{th}$ coefficient of the Dirichlet series for the logarithmic derivative of $L(s, \PSY{i}{j})$. Writing $c_{ij}$ for the coefficient of $\Lij(n)$ in the above expression for $\Lt(n)$, we have that
\[
	\SWAG = \sum_{\substack{p \leq x \\ p \nmid N}}
	\Lt(p) \cdot \left( p/x\right)^a \log\left(x/p\right).
\]

\subsection{Summing over All Prime Powers}\label{pp}
In order to apply the theory of symmetric power $L$-functions of an elliptic curve to estimate $\SWAG$, we must approximate $\SWAG$ with a sum over all prime powers less than or equal to $x$. To this end, consider the sums $\SWAG'$ and $\SWAG''$ given by
\begin{equation}\label{eq:speedy}
	\SWAG' \defeq \sum_{p \leq x}
	\Lt(p) \cdot \left( p/x \right)^a \log\left(x/p\right) \quad \text{and} \quad \SWAG'' \defeq
	\sum_{n \leq x} \Lt(n) \cdot \left( n/x \right)^a \log\left(x/n\right).
\end{equation}
We will first compute a bound on the difference $\SWAG' - \SWAG$. Since $|f_1(t)| \leq 8$ and $|f_2(t)| \leq 8$ for $t \in [-1,1]$, we have that $\left|\Lt(p^m)\right| \leq 64\log p$ for all prime powers $p^m$. So, we find that
$$\left|\SWAG' - \SWAG\right| \leq \sum_{\substack{p \leq x\\p \mid N}}\left|\Lt(p)\right| \cdot \left( p/x \right)^a \log\left(x/p\right) \leq  \sum_{p \mid N} 64\log p \cdot \log x \leq 64 \log N \log x.$$
We will next compute a bound on the difference $\SWAG'' - \SWAG'$. Observe that we have 
\[ \left|\SWAG'' - \SWAG'\right| \leq \sum_{\substack{p^m \leq x\\m \geq 2}}\left|\Lt(p^m)\right| \cdot \left( p^m/x \right)^a \log\left(x/p^m\right) \leq \sum_{\substack{p^m \leq x\\m \geq 2}}64 \log p \cdot \log\left(x/p^m\right). \]
It then follows readily from \cite{bach}*{(4.2)} that we have
\[ \left|\SWAG'' - \SWAG'\right| \leq 64 \left( 2.002\sqrt x + 4x^{1/3} \right). \]
Combining the above bounds, we find that
\begin{equation}
	\left|\SWAG'' - \SWAG\right| \leq 64\left( 2.002\sqrt x + 4x^{1/3} + \log N \log x \right).
	\label{eq:schrodie}
\end{equation}
We will now bound $\SWAG''$ to obtain an estimate for $\SWAG$.
\subsection{Reducing $\SWAG''$ to an Integral}
\label{subsec:main_est} Recall that by the definitions of $\Lt(n)$ and $\SWAG''$, \mbox{we have}
\[ \SWAG'' = \sum_{0\le i,j \le 2} c_{ij} \sum_{n \le x} \Lambda_{i \otimes j}(n) \cdot (n/x)^a \log (x/n).  \]
Taking $\SWAG_{ij}$ to denote the inner sum in the above equality, we can express $\SWAG_{ij}$ in the usual way (see Lemma 4.2 of~\cite{bachagain}) as a contour integral:
\begin{equation}\label{eq:ints}
\SWAG_{ij} = \frac{1}{2\pi i}\int_{2-i\infty}^{2+i\infty} \frac{x^s}{(s+a)^2} \left( -\frac{L'}{L}(s, \PSY ij) \right) \; ds.
\end{equation}
To evaluate the integral~\eqref{eq:ints}, we will utilize the Residue Theorem. Pick a large number $T > 0$ in such a way that $T$ does not coincide with the ordinate of a nontrivial zero of $L(s, \PSY{i}{j})$, and let $U > 0$ be a large number such that $-U$ does not coincide with any trivial zeros of $L(s, \PSY{i}{j})$.
Taking $\SWAG_{ij}(T)$ to be the truncation of $\SWAG_{ij}$ up to height $T$ and letting $R_{ij}$ be the sum over all residues of the integrand in \eqref{eq:ints} within the rectangle whose vertices are $2 \pm iT$
and $-U \pm iT$, we have by the Residue Theorem that
\[
	\SWAG_{ij}(T)
	= R_{ij} + \frac{1}{2\pi i}\int_{\Gamma_{T,U}} \frac{x^s}{(s+a)^2} \left( -\frac{L'}{L}(s, \PSY ij) \right) \; ds,
\]
where $\Gamma_{T,U}$ denotes the path defined by
\[
	2-iT
	\;\longrightarrow\; -U -iT
	\;\longrightarrow\; -U +iT
	\;\longrightarrow\; 2+iT.
\]
Observe that by the absolute convergence of the above integral, we have the following results: 
\begin{enumerate}
\item The integral over the horizontal legs of the contour can be made arbitrarily small by taking $T$ sufficiently large.
\item The integral over the vertical leg of the contour can be made arbitrarily small by taking $U$ sufficiently large. 
\end{enumerate}
Thus, we take $T \to \infty $ and $U \to \infty$, in which case the integral over the contour $\Gamma_{T,U}$ vanishes. The remainder of this section is devoted to estimating $R_{ij}$, which is now the sum of residues due to all zeros of $L(s, \PSY{i}{j})$ in the half-plane defined by $\sigma < 2$. 

\subsection{Estimating $R_{ij}$}\label{darkhorse}
Recall from Section~\ref{defs} that $L(s, \PSY{i}{j})$ has a pole at $s = 1$ if and only if $i = j = 0$. Thus, the possible pole at $s = 1$ contributes a residue of
\begin{equation}
	e_{i+j}\frac{x}{(1+a)^2}
	\label{eq:puru}
\end{equation}
The above residue gives rise to the ``main term'' in our estimate of $\SWAG''$. We next consider the residues that correspond to the zeros, both trivial and nontrivial, of $L(s, \PSY{i}{j})$.

\subsubsection{Residues at Trivial Zeros}\label{meanttofly}
Let $\zt_{ij}$ denote the collection of trivial zeros of the $L$-function $L(s, \PSY{i}{j})$. The sum of residues due to the trivial zeros is then given as follows:
$$\left|\sum_{\rho \in \zt_{ij}} \frac{x^\rho}{(\rho + a)^2}\right| = \sum_{\rho \in \zt_{ij}} \frac{x^\rho}{(\rho+a)^2},$$
where equality holds because $\zt_{ij} \subset \RR$. It can be deduced from our description of the local parameters at $\infty$ for the symmetric-power $L$-functions (see Section~\ref{kittypurry}) that $0 \in \zt_{ij}$ if and only if $i = j = 0$ and that $\max \big(\zt_{ij}\big) \leq -1/2$ if $i + j > 0$. We can therefore write the sum of residues due to trivial zeros as follows:
$$\frac{e_{i+j}}{a^2} + \sum_{\substack{\rho \in \zt_{ij} \\ \rho \neq 0}}\frac{x^\rho}{(\rho + a)^2} \leq \frac{e_{i+j}}{a^2} + (1-e_{i+j})d_{ij} \frac{x^{-1/2}}{(-1/2+a)^2} + d_{ij}\sum_{k = 1}^\infty \frac{x^{-2k}}{(2k-a)^2}.$$
Since $a \leq 1/4$, we can estimate the sum on the right-hand-side of the above inequality using the Cauchy-Schwarz Inequality as follows:
$$\sum_{k=1}^\infty \frac{x^{-2k}}{(2k-a)^2} \leq 16\sum_{k=1}^\infty \frac{x^{-2k}}{(8k-1)^2} \leq \frac{0.337}{\sqrt{x^4 - 1}}.$$
 Combining our results, we find that the residues due to nontrivial zeros are bounded by
$$\leq \frac{e_{i+j}}{a^2} + \frac{16(1-e_{i+j})d_{ij}}{\sqrt{x}} +
\frac{0.337d_{ij}}{\sqrt{x^4 - 1}}.$$
In what follows, we will take $\Upsilon = 0.337$ for convenience. Summing over $0 \leq i,j \leq 2$, we find that the residues due to nontrivial zeros contribute a total of
\begin{equation}
\frac{1}{a^2} + \frac{63 \cdot 16}{\sqrt{x}} + \frac{64\Upsilon}{\sqrt{x^4 - 1}}.
\label{eq:monster}
\end{equation}

\subsubsection{Residues at Nontrivial Zeros}
All of the nontrivial zeros lie on the line $\sigma = 1/2$ because we are assuming that GRH holds. The sum of these residues is bounded in absolute value as follows, where we write $\zn_{ij}$ for the collection of nontrivial zeros of $L(s, \PSY{i}{j})$:
\begin{equation}\label{eq:rasam}
	\left\lvert  \sum_{\rho \in \zn_{ij}} \frac{x^\rho}{(\rho+a)^2} \right\rvert
	\le \sqrt{x} \cdot \sum_{\rho \in \zn_{ij}} \frac{1}{\left\lvert \rho+a \right\rvert^2}.
\end{equation}
We now follow the method of proof used in Lemma 4.6 of \cite{bach}. Since the completed $L$-function $\Lambda(s, \PSY{i}{j})$ defined in~\eqref{eq:compsies} is entire, it can be represented as a Hadamard product.
Setting this Hadamard product equal to the right-hand-side of~\eqref{eq:compsies}, logarithmically differentiating, and applying self-duality yields that
\begin{eqnarray*}
-\dLL(1+a) & = & \half\log\Nij + \frac {e_{i+j}}{a} + \frac {e_{i+j}}{a+1} + \dgij(1+a) + \label{h1}\\
& & B_{ij} - \sum_{\rho \in \zn_{ij}} \left(\frac{1}{\rho} + \frac{1}{1 + a - \rho}\right), \nonumber
\end{eqnarray*}
where $B_{ij}$ is such that $\Re B_{ij} = \Re \sum_\rho \frac{1}{\rho}$. From this, along with the self-duality of $L(s, \PSY{i}{j})$, we deduce from the above equality that
\begin{eqnarray}
-\dLL(1+a) & = & \half\log\Nij + \frac {e_{i+j}}{a} + \frac {e_{i+j}}{a+1} + \dgij(1+a) - \label{h2}\\
& & \frac{1}{2} \sum_{\rho \in \zn_{ij}} \left(\frac{1}{1 + a - \rho} + \frac{1}{1 + a - \overline{\rho}}\right).\nonumber
\end{eqnarray}
Finally, observe that we can bound the logarithmic derivative of $L(s, \PSY{i}{j})$ in terms of $\zeta(s)$ as follows:
\begin{equation}\label{h3}
\left|\dLL(1+a)\right|
\le \left\lvert \sum_{n \ge 1} \Lambda_{i \otimes j}(n) n^{-s} \right\rvert
\le d_{ij}\left|\frac{\zeta'}{\zeta}(1+a)\right|.
\end{equation}
Then, using the above results as well as the identity
\begin{equation}\label{h4}
	\sum_{\rho \in \zn_{ij}} \frac{1}{|\rho+a|^2}
	= \frac{1}{2a+1} \sum_{\rho \in \zn_{ij}} \left( \frac{1}{1+a-\rho} + \frac{1}{1+a-\overline\rho} \right),
\end{equation}
we can estimate the sum over the nontrivial zeros on the right-hand-side of~\eqref{eq:rasam}. This estimation is carried out explicitly in Section~\ref{subsec:critical}. Summing over $0 \leq i,j \leq 2$, we find that the contribution due to the residues at nontrivial zeros is bounded in absolute value by
\[
	\sum_{0 \le i,j \le 2} |c_{ij}| \sum_{\rho \in \zn_{ij}} \frac{\sqrt x}{|\rho+a|^2} <
	\left( \frac{32}{2a+1} \log N + A_2 \right) \sqrt x,
\]
where $A_2 \asymp 1/a$ is a constant depending only on $a$. We now estimate the residue due to the double pole at $s = -a$.

\subsubsection{Residue at $s = -a$.} From the Cauchy Integral Formula, we find that after summing over $0 \leq i,j \leq 2$, the residue at $s = -a$ contributes a total of
$$ < x^{-a}
	\left( (\log x)\left( \frac{32(a+2)}{2a+1}\log N + A_6 \right)
	+ \frac{32}{2a+1}\log N + A_4 \right),$$
where $A_4, A_6 \asymp 1/a$ are constants depending only on $a$. We defer the explicit details of the estimate to Section~\ref{breakfree}.
\section{Explicit Estimates for Theorem~\ref{thm:main}}
\label{sec:pain}

In this section, we provide detailed proofs of the explicit estimates that are required to complete the proof of Theorem~\ref{thm:main} (see Sections~\ref{darkhorse}). In what follows, we take our parameter $a$ with $0 < a < 1/4$.

\subsection{Explicit Sum over Nontrivial Zeros}
\label{subsec:critical}
Applying the results~\eqref{h2},~\eqref{h3}, and~\eqref{h4} stated in Section~\ref{darkhorse},  we can estimate the sum on the right-hand-side of~\eqref{eq:rasam} in the following way:
\begin{align*}
	\sum_{\rho \in \zn_{ij}} \frac{1}{|\rho+a|^2}
    & = \frac{1}{2a+1} \sum_{\rho \in \zn_{ij}} \left( \frac{1}{1+a-\rho} + \frac{1}{1+a-\overline\rho} \right)\\
	&\le \frac{2}{2a+1} \left( 
    	\half\log\Nij + \frac {e_{i+j}}{a} + \frac{e_{i+j}}{a+1} + \left|\dgij(1+a)\right|
        + d_{ij} \left\lvert \frac{\zeta'}{\zeta}(1+a) \right\rvert \right). \\
	\intertext{Then along with our digamma estimate from Lemma~\ref{lem:digamma} gives}
    &\le \frac{2}{2a+1} \left( \half\log\Nij + \frac {e_{i+j}}{a} + \frac{e_{i+j}}{a+1}
    	+ d_{ij}A_1 + 1\right),
\end{align*}
where for convenience, we define the constant $A_1$ by
\[ 
	A_1 \defeq \half\left(\tau + \log \left( \frac{a}{2}+ \frac{7}{2} \right) \right)
    + \frac 1a + \gamma.
\]
(Here, we have also used the fact that
$ \left\lvert \frac{\zeta'}{\zeta}(1+a) \right\rvert < \frac 1a + \gamma$
where $\gamma < 0.57722$ is the Euler-Mascheroni constant.)
Combining our results, we conclude that the contribution due to nontrivial zeros can be estimated is
\[ 
	\leq \sqrt{x} \cdot \frac{2}{2a+1} \left( \half\log\Nij
	+ \frac {e_{i+j}}{a} + \frac{e_{i+j}}{a+1} + d_{ij}A_1\right).
\]
Now, we sum the above bound over all $i,j \in \{0,1,2\}$. To do so, first notice that we have the following estimates/equalities:
\begin{align*}
	\sum_{0 \le i,j \le 2} |c_{ij}| \cdot \log \Nij
	&\leq  \sum_{0 \le i,j \le 2} |c_{ij}| \cdot \left( i(j+1) \log N_{E_1} + j(i+1) \log N_{E_2} \right) =  32\log N, \\
    \sum_{0 \le i,j \le 2} |c_{ij}| \cdot d_{ij}A_1 &= A_1 \sum_{0 \le i,j \le 2} |c_{ij}| \cdot (i+1)(j+1) = 64A_1, \text{ and}\\
    \sum_{0 \le i,j \le 2} |c_{ij}| \cdot 
    \left( \frac {e_{i+j}}{a} + \frac{e_{i+j}}{a+1} \right)
	&= \frac 1a + \frac1{a+1} = \frac{2a+1}{a(a+1)}.
\end{align*}
Given these results, we conclude that the contribution due to nontrivial zeros is bounded by
\begin{equation}
	\sum_{0 \le i,j \le 2} |c_{ij}|
	\sum_{\rho \in \zn_{ij}} \frac{\sqrt x}{|\rho+a|^2}
	\leq \left( \frac{32}{2a+1} \log N + A_2 \right) \sqrt x,
	\label{eq:blame}
\end{equation}
where for convenience we put
\[ A_2 \defeq \frac{2}{2a+1} \left( 64A_1 + \frac{2a+1}{a(a+1)} \right)
	= \frac{128}{2a+1} A_1 + \frac{2}{a(a+1)}. \]

\subsection{The Residue at $s = -a$}\label{breakfree}
For convenience, let $L_{i\otimes j}(s) \defeq L(s,  \PSY ij)$. We will first need to estimate $\left(-\dLL\right)'(s)$ at $s=-a$.
Recall that we can apply (5.24) from~\cite{iwaniec} to obtain the following equality:
\begin{eqnarray}
	-\dLL(s) &=& -\dLL(2) + \dgij(s)-\dgij(2)
	+ e_{i+j} \left( \frac{1}{s} + \frac{1}{s-1} - \frac32 \right) + \label{eq:dLL} \\
	& &  \sum_{\rho \in \zn_{ij}} \left( \frac{1}{2-\rho}-\frac{1}{s-\rho} \right).  \nonumber
\end{eqnarray}

\begin{lemma}
	\label{lem:difflogdiff}
	For any $a \in (0,1/4]$, we have that
	\[ \sum_{0 \le i,j \le 2} |c_{ij}| \left| \left( -\dLL\right)'(-a) \right|
		\le \frac{32}{2a+1}\log N + A_4 \]
      where $A_4 \asymp 1/a$ is a constant depending only on $a$.
\end{lemma}
\begin{proof}
	We begin by differentiating \eqref{eq:dLL} to obtain
	\[ \left( -\dLL\right)'(-a)
		= \left(\dgij\right)'(-a)
		-e_{i+j}\left[ \frac{1}{a^2} + \frac{1}{(a+1)^2} \right]
		+ \sum_{\rho \in \zn_{ij}} \frac{1}{(a+\rho)^2}.  \]
	It clearly suffices to estimate the first term in the above equality, which is given by the following sum:
	\begin{equation}\label{Benfordz} \frac14 \sum_\kappa \ddG \left( \frac{-a+\kappa}{2} \right) \end{equation}
	where $\kappa$ ranges over the local roots of $L(s, \PSY ij)$. We recall that $\ddG(s)$ is decreasing for $s \in (0,\infty)$.
	Moreover, by differentiating \eqref{eq:starshipsyo}, we have that
	\[ \ddG\left(-\frac a2\right) = \frac{4}{a^2} + \ddG\left(1-\frac a2\right). \]
    To obtain an estimate on~\eqref{Benfordz}, we can use the above functional equation to rewrite the term corresponding $\kappa = 0$ so that it has positive argument. Then, we can use the monotonicity of $\ddG(s)$ on $s \in (0, \infty)$ and the assumption that $a \leq 1/4$ to upper bound the resulting $\ddG$ terms, all of which have positive argument.
As a representative example of how this estimate is performed, the term corresponding to $\kappa = 0$ is bounded as follows:
$$\ddG\left(1-\frac a2\right) \leq \ddG\left(1-\frac 18\right) \leq 2.006.$$
    Using this procedure to estimate~\eqref{Benfordz} and summing the result over $0 \leq i,j \leq 2$ yields the following bound:
	\begin{align*}
		\sum_{0\le i,j \le 2} |c_{ij}| \sum_\kappa \ddG\left( \frac{-a+\kappa}{2}  \right)
		&\le \frac{4}{a^2} + 426.875.
	\end{align*}
    For brevity, we heretofore let $\Psi \defeq 426.875$; an explicit computation of $\Psi$ is performed in the ancillary program file (see Appendix~\ref{tabwars} regarding this file).
	We then have that
	\[ \sum_{0 \le i,j \le 2} |c_{ij}| \left| \left( -\dLL\right)'(-a) \right|
		\le \frac{1}{4} \left( \frac{4}{a^2} + \Psi \right)
		+ \left[ \frac{1}{a^2} + \frac{1}{(a+1)^2} \right]
		+ \frac{32}{2a+1} \log N + A_2.  \]
	Taking $A_4 \defeq \Psi/4 + \frac{2}{a^2} + \frac{1}{(a+1)^2} + A_2$ yields the lemma.
\end{proof}
We now make the following two estimates:
\begin{lemma}\label{lem:yomommastacos}
	For any $a \in (0,1/4]$, we have that
	\begin{equation}\label{goharkereagles} \sum_{0 \leq i,j \leq 2} |c_{ij}| \left( -\dLL(-a) \right)
		\le \frac{32(a+2)}{2a+1}\log N + A_6, \end{equation}
	where $A_6 \asymp 1/a$ is a constant depending only on $a$.
\end{lemma}
\begin{proof}
    Notice that we have the following inequalities:
	\begin{eqnarray*} \left|\sum_{0 \leq i,j \leq 2} |c_{ij}|\sum_{\rho \in \zn_{ij}} \left( \frac{1}{2-\rho}-\frac{1}{-a-\rho} \right)\right|
		& \leq & (a+2)\sum_{0 \leq i,j \leq 2} |c_{ij}|\sum_{\rho \in \zn_{ij}} \frac{1}{\left\lvert \rho+a \right\rvert^2} \\
		& \le & (a+2)\left(\frac{32}{2a+1} \log N + A_2\right),
        \end{eqnarray*}
	where the last estimate follows from the results of Section~\ref{subsec:critical}. Now, applying Lemma~\ref{lem:digamma} to estimate the remaining terms in~\eqref{eq:dLL}, we see that the left-hand-side of~\eqref{goharkereagles} is at most
	\[
		\frac{32(a+2)}{2a+1} \log N +(a+2) A_2
		+ \sum_{0 \leq i,j \leq 2} |c_{ij}| d_{ij} \eta
		+ \sum_{0 \leq i,j \leq 2} |c_{ij}|\left(6+\frac{e_{i+j}}{a}\right)
		+ A_5,
	\]
	where $\eta$ and $A_5$ are defined as follows:
	\[ \eta \defeq \log 12 +\tau+\zzt \quad \text{and} \quad
		A_5 \defeq \frac1a+\frac1{1+a}+\frac32. \]
	Here, the $\log 12 + \tau$ term arises from applying Lemma~\ref{lem:digamma} to estimate the two digamma terms in~\eqref{eq:dLL}, and the $d_{ij}\zzt$ term is an easy estimate for $\left|\dLL(2)\right|$. Finally, if we define $A_6$ by
	\[ A_6 \defeq (a+2) A_2 + 64\eta + 96 + \frac{1}{a} + A_5, \]
	then we obtain the lemma.
\end{proof}

The contribution of the residue at $s=-a$ is given by
\[ \left(\log x \cdot \left( -\dLL(-a) \right) 
	+ \left( -\dLL \right)' (-a)\right)x^{-a}.  \]
Thus, when we sum over all $0 \leq i,j \leq 2$, Lemmas~\ref{lem:difflogdiff} and~\ref{lem:yomommastacos} tell us that the contribution due to the residue at $s=-a$ is bounded by
\[ \le  x^{-a}
	\left( (\log x)\left( \frac{32(a+2)}{2a+1}\log N + A_6 \right)
	+ \frac{32}{2a+1}\log N + A_4 \right). \]
Since $x, a > 0$ we can simplify this to an upper bound of
\begin{equation}
	\left( (\log x)\left( \frac{32(a+2)}{2a+1}\log N + A_6 \right)
	+ \frac{32}{2a+1}\log N + A_4 \right).
	\label{eq:orion}
\end{equation}

\subsection{The Final Estimate}
In this section, we obtain the explicit bound stated in Theorem~\ref{thm:main}.
Collating the main term in \eqref{eq:puru} with \eqref{eq:schrodie}, \eqref{eq:monster}, \eqref{eq:orion}, \eqref{eq:blame}, 
it suffices to find an $x$ such that
\begin{eqnarray*}
	\frac{x}{(1+a)^2} & > & \left( \frac{32}{2a+1} \log N + A_3 \right) \sqrt x + 64\left( 4x^{1/3} + \log N \log x \right) + \\
	& &  \frac{1}{a^2} + \frac{63 \cdot 16}{\sqrt x} + \frac{64\Upsilon}{\sqrt{x^4-1}} + \\
	& &  (\log x)\left( \frac{32(a+2)}{2a+1}\log N + A_6 \right) + \frac{32}{2a+1}\log N + A_4.
\end{eqnarray*}
where for convenience we have defined $A_3\defeq A_2 + 64 \cdot 2.002$. We now select
\[ a = \frac14 \sqrt[3]{\frac{\log 121}{\log N}} = \frac{C_0}{\sqrt[3]{\log N}} \]
where we define $C_0 \defeq \sqrt[3]{\log 121}/4$. It is clear that $a \leq 1/4$, because we have that $N_{E_1}, N_{E_2} \geq 11$, so $N \ge 121$.

Now, we seek to compute $(1+a)^2\left( \frac{32}{2a+1}\log N + A_3 \right)$.
First, we have $\frac{(1+a)^2}{2a+1} \le 1+a^2$,
which takes care of the first part.
As for $(1+a)^2 A_3$, 
we use the fact that $A_1 \asymp 1/a$ in the limit as $a \to 0$ in order to simplify our computations. Specifically, we will isolate the $1/a$ term in $A_1$ and bound the remaining dependence on $a$ using the fact that $0 < a \leq 1/4$.
Notice that we have
\[
	A_1 \le \frac1a + C_1 \quad\text{where}\quad
	C_1 \defeq \half \left(\tau +\log \frac{29}{8} \right) + \gamma.
\]
Next, since we took $A_3 = 64 \cdot 2.002 + A_2$, we have that
\[ A_3 = 64 \cdot 2.002 + \frac{128}{2a+1}A_1 + \frac{2}{a(a+1)}. \]
thus
\[ (1+a)^2 A_3 = 64\cdot 2.002(1+a)^2 + \frac{128(1+a)^2}{2a+1}A_1
+ 2\left( \frac 1a+1 \right). \]
The middle term is bounded by
\[
	\frac{(1+a)^2}{2a+1}A_1
	\le \frac{(1+a)^2}{a(2a+1)}
	+ \frac{(1+a)^2}{2a+1} C_1
	\le \frac 1a + \frac{a}{2a+1}
	+ \frac{(1+a)^2}{2a+1} C_1
	\le \frac 1a + \half + \frac{25}{24} C_1.
\]
Thus, we finally deduce that
\[ (a+1)^2 A_3 \le \frac{130}{a} + C_2, \]
where $C_2$ is defined as follows:
\[ C_2 \defeq 64 \cdot 2.002 \cdot \frac{25}{16} + 64 + 128\cdot \frac{25}{24} C_1 + 2.  \]

Therefore now suffices to find $x$ such that the following inequality holds:
\begin{eqnarray}
	x &>&  \left( 32\log N + 32a^2\log N + \frac{130}{a} \right)\sqrt x + \label{vinogradova} \\
	& & C_2 \sqrt x + (1+a)^2 \cdot 64\left( 4x^{1/3} + \log N \log x \right) + \nonumber \\
	& & (1+a)^2 \left( \frac{1}{a^2} + \frac{63 \cdot 16}{\sqrt x} + \frac{64\Upsilon}{\sqrt{x^4-1}} \right) + \nonumber \\
	& & (1+a)^2 \left( (\log x)\left( \frac{32(a+2)}{2a+1}\log N + A_6 \right) + \frac{32}{2a+1}\log N + A_4 \right). \nonumber
\end{eqnarray}
A computer calculation shows that if $x \ge \left( 32\log N + 32a^2\log N + \frac{130}{a} + C_2 \right)^2$, then the last three lines of the above sum are at most $C_3 \sqrt x \log x$ for $C_3 = 56.958$. So, it suffices that our choice of $x$ satisfies
\begin{equation}\label{j}
\sqrt x > 32\log N + C_4\sqrt[3]{\log N} + C_3 \log x 
\end{equation}
where $C_4 \defeq 32C_0^2 + 130/C_0 \leq 314.042$. 
Choosing 
\begin{equation}
	x= \left( 32\log N + 315\sqrt[3]{\log N} + 533 \log \log N \right)^2,
	\label{eq:explicit_main}
\end{equation}
we obtain Theorem~\ref{thm:main} in its completely explicit form. The computation of the various constants used in the above proof of Theorem~\ref{thm:main} can be found in the ancillary program file (see Appendix~\ref{tabwars} regarding this file).
\section{Proof of Theorem~\ref{thm:roar}}\label{likith}

In this section, we prove Theorem~\ref{thm:roar}. Once again, the method we employ is based in part on the work of Bach and Sorenson (see~\cite{bach}).

Let $E$ be a non-CM elliptic curve over $\QQ$ with squarefree conductor $N_E$. Recall again that for every prime $p$, we can associate an angle $\theta_{p} \in [0, \pi]$ to the curve $E$, such that the trace of Frobenius at $p$ is expressible as $a_{E}(p) = 2\sqrt{p}\cos\theta_{p}$. Now, fix a subinterval $I = [\alpha, \beta] \subset [0,\pi]$, and denote by $\mu$ the Sato-Tate measure of $I$. Suppose that we have a function $f : \RR \to \RR$ such that $f(\theta) \leq 0$ when $\theta \in [0,\pi] \setminus I$. Then consider the sum $\SWAG$ defined as follows:
\begin{equation}\label{eq:feds2}
	\SWAG \defeq
	\sum_{\substack{p \leq x \\ p \nmid N_E}}\, (\log p) f(\theta_{p}) \cdot \left( p/x\right)^a \log\left(x/p\right),
\end{equation}
where $a \in (0,1/4]$ is a fixed constant. If the sum $\SWAG$ is strictly positive, then there must be a prime $p \le x$ such that $\theta_p \in I$. Thus, to prove Theorem~\ref{thm:roar}, we will pick $x$ in such a way that $\SWAG > 0$.

\subsection{Choosing the Function $f$}\label{longlistofexlovers}
We must first define a suitable choice of $f$. Fix a positive integer $M \ge 8$, and consider the Beurling-Selberg minorant $S : \RR \to \RR$ defined by
\begin{eqnarray}
S(\theta) & = & \frac{\beta - \alpha}{2\pi} -
\sum_{k = -M}^M \left(1-\frac{|k|}{M+1}\right)\frac{e^{i k(\beta-\theta)} + e^{i k(\theta - \alpha)}}{2(M+1)} - \label{bs}\\
& & \sum_{k = 1}^M g\left(\frac{k}{M+1}\right) \frac{e^{ik(\beta -\theta)} - e^{- ik (\beta - \theta)} + e^{ik(\theta -\alpha)} - e^{-ik (\theta - \alpha)}}{2i(M+1)}, \nonumber
\end{eqnarray}
where we define the function $g$ by \[ g(u) = -(1-u)\cot \pi u - 1/\pi. \]
As described in Section 1.2 of~\cite{montgomery}, the minorant $S$ is periodic modulo $2\pi$, and it satisfies the following important properties:
\begin{enumerate}
	\item For all $\theta \in [0,2\pi]$, we have that $S(\theta) \leq \chi_I(\theta)$, where $\chi_I$ denotes the indicator function of $I$. In particular, we have that $S(\theta) \leq 0$ for $\theta \not\in I$.
	\item If $\hat{S}(n)$ denotes the $n^\mathrm{th}$ Fourier series coefficient of $S$, then we have that $\hat{S}(n) + \hat{S}(-n) = 2\Re \hat{S}(n)$ for all $n$.
    \item $S$ can be expressed as follows:
		\[ S(\theta) = \frac{\beta - \alpha}{2\pi} - B\left(\frac{\beta - \theta}{2\pi}\right) - B\left(\frac{\theta - \alpha}{2\pi}\right), \]
    where $B(x)$ is Beurling's polynomial, defined by
	\[ B(x) \defeq V(x) + \frac{1}{2(M+1)}\Delta(x), \]
	where $V(x)$ is Vaaler's polynomial and $\Delta(x)$ is Fejer's kernel; these functions are defined by
	\begin{align*}
		V(x) &\defeq \frac{1}{M+1}\sum_{k = 1}^M g\left(\frac{k}{M+1}\right) \sin 2 \pi k x \\
		\Delta(x) &\defeq \sum_{k = -M}^{M} \left(1 - \frac{|k|}{M}\right)e^{2\pi i k x}.
	\end{align*}
	Vaaler's polynomial is a good approximation to the sawtooth function $s(x)$ whose graph on the period $[0,1]$ is given by $s(x) = x - 1/2$; in particular, $|V(x)| \leq |s(x)| \leq 1/2$, a result known as Vaaler's Lemma. One can use this bound on $V(x)$ to check that $|S(x)| \leq 5/2$.
\end{enumerate}
 Now, define $f$ by $f(\theta) \defeq S(\theta) + S(-\theta)$. From property (1), we readily deduce that $f(\theta) \leq 0$ when $\theta \in [0,\pi] - I$, as desired. Moreover, from~\eqref{bs}, one can deduce that $f$ is given in terms of the Chebyshev polynomials of the second kind by the formula
$$f(\theta) = \sum_{k = 0}^M \Xi_k U_k(\cos \theta),$$
where the coefficients $\Xi_k$ can be computed using property (2). Doing so, we find that if $e_{k,\ell} = 1$ when $k = \ell$ and $0$ otherwise, then
\begin{align*}
	\Xi_k & \defeq \frac{\beta - \alpha}{\pi}e_{k,0} - \left(1 - \frac{k}{M+1}\right)\frac{\cos k\alpha + \cos k \beta}{M+1} - g\left(\frac{k}{M+1}\right)\frac{\sin k \beta  - \sin k \alpha}{M+1} + \\
	& \left(1 - \frac{k+2}{M+1}\right)\frac{\cos (k+2)\alpha + \cos (k+2) \beta}{M+1} + g\left(\frac{k+2}{M+1}\right)\frac{\sin (k+2) \beta  - \sin (k+2) \alpha}{M+1},
\end{align*}
with the terms on the second line of the above equality being omitted when $k = M-1$ or $k = M$.
From part (2) of Lemma 3.2 of~\cite{rouse}, we have that
\begin{equation}\label{pucked}
\left\lvert \Xi_0 - \mu \right\rvert 
=
\left\lvert -\frac{2}{M+1} + \frac{M-1}{M+1} \frac{\cos2\alpha+\cos2\beta}{M+1} \right\rvert
\le \frac{4}{M+1}.
\end{equation}
As for $k \ge 1$, observe that we have the following estimate:
\begin{equation}\label{gangnamstyle}
	|\Xi_k| \le \frac{2}{M+1}
	\left( 2 - \frac{2k+2}{M+1}
	+ \frac{e_{k,M}}{M+1}
	+ \left\lvert g\left( \frac{k}{M+1} \right) \right\rvert
	+ \left\lvert g\left( \frac{k+2}{M+1} \right) \right\rvert \right),
\end{equation}
which is valid even when $k = M-1$ and $k = M$ (but not when $k=0$), as long as we take $g(u) = 0$ for $u \ge 1$.

\begin{remark}
In the case when our interval $I \subset [0,\pi]$ has either $0$ or $\pi$ as an endpoint, the minorant $f(\theta)$ can be taken to be the Beurling-Selberg minorant of the connected interval $I \cup I'$, where $I'$ is the reflection of $I$ across the aforementioned endpoint. This choice of $f$ (as opposed to taking $f(\theta) = S(\theta) + S(-\theta)$) would decrease the coefficient $C$ of the main term in our final bound.
\end{remark}

\subsection{Bounds on Sums of $\Xi_k$}
We now make the following computations, which are critical to our proof of Theorem~\ref{thm:roar}:
\begin{lemma}\label{callmemaybe}
	Assume $M \ge 8$.
	We have that
	\begin{align}
		\sum_{k = 0}^M |\Xi_k| &\le \frac2\pi\log M + \frac{21}{5}, \label{eq:seenoevil} \\
		\sum_{k = 0}^M k|\Xi_k| &\le \frac{CM}{16} \label{eq:hearnoevil}, \text{ and} \\
		\sum_{k = 0}^M (k+1)|\Xi_k| &\le \frac{CM}{16}  + \frac2\pi \log M + \pi, \label{eq:speaknoevil}
	\end{align}
	where we take $C = 32(1/3+1/\pi).$
\end{lemma}
\begin{proof}
	We will begin by doing the summation for $k \ge 1$, and we will treat the $k=0$ term separately.
	A straightforward computation yields that the first three terms of~\eqref{gangnamstyle} sum to the following:
	\begin{align*}
		\sum_{k = 1}^M \left( 2 - \frac{2k+2}{M+1} + \frac{e_{k,M}}{M+1} \right)
		&= \frac{M^2-M+1}{M+1}, \text{ and} \\
		\sum_{k = 1}^M k \left( 2 - \frac{2k+2}{M+1} + \frac{e_{k,M}}{M+1} \right)
		&= \frac{M^2-M}{3} + \frac{M}{M+1}.
	\end{align*}
	The sum of these two terms is given by $$\sum_{k = 1}^M (k+1) \left( 2 - \frac{2k+2}{M+1} + \frac{e_{k,M}}{M+1} \right) = \frac{M^2+2M-3}{3} + \frac{2}{M+1}.$$
As for summing the two terms of~\eqref{gangnamstyle} involving $g$, we first observe that
	\[
		 \frac{2}{M+1} \sum_{k=1}^{M} \left\lvert g\left( \frac{k+2}{M+1} \right) \right\rvert
		 \leq \frac{2}{M+1} \sum_{k=1}^{M} \left\lvert g\left( \frac{k}{M+1} \right) \right\rvert
	\]
	as $|g(u)|$ is monotonically decreasing on the interval $(0,1)$.
	Similar inequalities are true with the weights of $k$ and $k+1$ added in.
	Moreover, $g(u) \le 0$ for all $u \in (0,1)$, so we may write
	\[
		\sum_{k=1}^{M} \left\lvert g\left( \frac{k}{M+1} \right) \right\rvert
		=
		\sum_{k=1}^{M}
		\left(\left(1 - \frac{k}{M+1}\right)\cot\left(\frac{\pi k}{M+1}\right) + \frac{1}{\pi}\right)
	\]
	and again, the analogous equalities hold with coefficients of $k$ and $k+1$.
	Since the $\frac1\pi$ terms are easy to sum, the main content is in estimating the sum of cotangents.
	First, notice that
	\[  \sum_{k=1}^{M} \frac{k}{M+1}\left(1 - \frac{k}{M+1}\right)\cot\left(\frac{\pi k}{M+1}\right) = 0, \]
	which follows readily from the fact that the function $u(1-u)\cot(\pi u)$
	is antisymmetric on the interval $(0,1)$ about the point $u = 1/2$.
	Also, we note that for $u < 1/2$, we
	have the inequality $u(1-u)\cot(\pi u) < 1/\pi$ and that
	for $u \ge 1/2$ we have $u(1-u)\cot(\pi u) < 0$.
	It then follows from the bound on the harmonic numbers given by~\cite{lilytsai} that we have
	\[ 
		\sum_{k=1}^{M} \frac{1}{M+1}\left(1 - \frac{k}{M+1}\right)\cot\left(\frac{\pi k}{M+1}\right) 
		< \sum_{k=1}^{M/2} \frac{1}{\pi k}
		< \frac1\pi \left( \log(M/2) + \gamma + \frac{1}{M+\frac13} \right).
	\]
We then conclude that we have the following two inequalities:
	\[
		\frac{2}{M+1} \sum_{k=1}^{M} \left\lvert g\left( \frac{k+2}{M+1} \right) \right\rvert
		\leq \frac{2}{M+1} \sum_{k=1}^{M} \left\lvert g\left( \frac{k}{M+1} \right) \right\rvert
		< \frac1\pi\left[\frac{2M}{M+1}+ \log(M/2) + \gamma + \frac{1}{M+\frac13}\right],
	\]
	\[
		\text{and} \quad \frac{2}{M+1} \sum_{k=1}^{M} k\left\lvert g\left( \frac{k+2}{M+1} \right) \right\rvert
		\leq \frac{2}{M+1} \sum_{k=1}^{M} k\left\lvert g\left( \frac{k}{M+1} \right) \right\rvert
		= \frac{M}{\pi}.
	\]
	We can now estimate the left-hand-side of \eqref{eq:hearnoevil}, since there is no $|\Xi_0|$ term. In this case, we obtain the following bound:
	\begin{equation}
		\frac{2}{M+1} \left( \frac{M^2-M}{3} + \frac{M}{M+1} \right) + \frac{2M}{\pi},
		\label{eq:TST}
	\end{equation}
	which is enough to imply \eqref{eq:hearnoevil} when $M \ge 8$. For $|\Xi_0|$ we use a naive estimate $|\Xi_0| \le 1 + \frac{4}{M+1}$ from~\eqref{pucked}.
	It follows that for~\eqref{eq:seenoevil}, we obtain the upper bound
	\begin{equation}
		\left( 1 + \frac{4}{M+1} \right)
		+  \frac{2(M^2-M+1)}{(M+1)^2} 
		+ \frac{2}{\pi}\left( \frac{2M}{M+1}+ \log(M/2) + \gamma + \frac{1}{M+\frac13} \right).
		\label{eq:IMO}
	\end{equation}
	One can check that~\eqref{eq:IMO} is bounded by $\frac{21}{5} + \frac2\pi \log M$ for $M \ge 8$. Finally, if we sum \eqref{eq:TST} and \eqref{eq:IMO}, we obtain \eqref{eq:speaknoevil}, again for $M \ge 8$.
	(The appearance of $\pi$ is purely coincidental in the right-hand side of \eqref{eq:speaknoevil};
	the optimal constant is approximately $3.14$.)
\end{proof}

\subsection{Summing over All Prime Powers}
Now that we have defined $f$, we return to our consideration of the sum $\SWAG$ defined in~\eqref{eq:feds2}. Consider the function $\Lt(n)$ defined as follows: if $n$ is not a prime power, then $\Lt(n) \defeq 0$, and if $n = p^m$, then $\Lt(n)$ is given as follows:
\begin{equation*}
	\Lt(n) \defeq (\log p)f(m\theta_{p}) = \sum_{k = 0}^M \Xi_k\Lambda_k(n),
\end{equation*}
where we recall from Section~\ref{defs} that $\Lambda_k(n)$ is the $n^\mathrm{th}$ coefficient of the Dirichlet series for the logarithmic derivative of $L(s, \Sym^k E)$. Now, writing $\SWAG$ in terms of $\Lt(n)$, we have that
\[
	\SWAG = \sum_{\substack{p \leq x \\ p \nmid N}}
	\Lt(p) \cdot \left( p/x\right)^a \log\left(x/p\right).
\]
As in Section~\ref{pp}, we want to approximate $\SWAG$ with an analogous sum over all prime powers up to $x$, for this will allow us to estimate $\SWAG$ using the analytic theory of symmetric-power $L$-functions. For this purpose, we introduce the sum $\SWAG'$, given by
\begin{equation}\label{eq:speedy1}
	 \SWAG' \defeq
	\sum_{n \leq x} \Lt(n) \cdot \left( n/x \right)^a \log\left(\frac{x}{n}\right).
\end{equation}
As discussed in Section~\ref{longlistofexlovers}, we have that $|S(\theta)| \leq 5/2$ for all $\theta \in \RR$, so we deduce that $|f(\theta)| \leq 5$ for all $\theta \in \RR$. Using the above bound, we can repeat the argument presented in Section~\ref{pp} to find that
\begin{equation}\label{starbuckslovers} 
	\SWAG' - \SWAG \leq 5\left(2.002\sqrt x + 4x^{1/3} + \log N_E \log x\right).
\end{equation}
In what follows, we will estimate $\SWAG'$.

\subsection{Estimating $\SWAG'$}
Expanding $\SWAG'$ in terms of the definition of $\Lt(n)$, we have that
\[ \SWAG' = \sum_{k=0}^M \Xi_k \sum_{n \le x} \Lambda_{k}(n) \cdot (n/x)^a \log (x/n).  \]
Let $\SWAG_{k}$  denote the inner sum in the above equality, weighted by the coefficient $\Xi_k$. We can repeat the arguments of Section~\ref{subsec:main_est} to express $\SWAG_k$ as an integral:
\begin{equation*}
\SWAG_k = \Xi_k \cdot \frac{1}{2 \pi i} \int_{2-i\infty}^{2+i\infty} \frac{x^s}{(s+a)^2} \left( -\frac{L'}{L}(s, \Sym^k E) \right) \; ds.
\end{equation*}
To estimate the above integral, we resort to the Residue Theorem as in Section~\ref{subsec:main_est}. By pushing our contour to $\sigma = -\infty$, we find that $\SWAG_k = R_k$, where $R_{k}$ is the sum of the integrand's residues in the half-plane defined by $\sigma < 2$. The remainder of this section is devoted to estimating $R_k$.

The main term, which is contributed by the pole at $s=1$ for $k=0$, is given as follows:
\begin{equation}\label{rirrian} |\Xi_0| \cdot \frac{x}{(1+a)^2}. \end{equation}
We next estimate the contribution due to residues due to all zeros of $L(s, \Sym^k E)$.

\subsubsection{Residues at All Zeros}\label{johnbarry}
By repeating the arguments used in Section~\ref{meanttofly},
we deduce that the contribution of residues from trivial zeros of $L(s, \Sym^k E)$ is bounded by
\[ \leq\frac{e_{k}}{a^2} + \frac{16(1-e_{k})(k+1)}{\sqrt{x}}  +  \frac{\Upsilon(k+1)}{\sqrt{x^4 - 1}}.\]
Using Lemma~\ref{callmemaybe} and \eqref{pucked} to sum over all $k$ with the weights $|\Xi_k|$, we obtain a total contribution from trivial zeros that is bounded by
\[
	\leq
	\frac{13}{9a^2}
	+ \left( \frac{CM}{16} + \frac{2}{\pi}\log M + \pi \right)\left( \frac{16}{\sqrt x} + \frac{\Upsilon}{\sqrt{x^4-1}}  \right).
\]
In analogy with the arguments of Section~\ref{subsec:critical}, the contribution of residues from the nontrivial zeros of $L(s, \Sym^k E)$ is bounded as follows for $M \geq 8$ and $k \le M$:
\begin{align*}
	&\le \sum_\rho \frac{\sqrt{x}}{|\rho+a|^2} \\
	&\le  \frac{2\sqrt{x}}{2a+1} \left(
		\half\log N_E^k + \frac {e_k}{a} + \frac{e_k}{a+1} + \left|\dgk(1+a)\right|
		+ (k+1) \left\lvert \frac{\zeta'}{\zeta}(1+a) \right\rvert \right) \\
	&\le \frac{2\sqrt{x}}{2a+1} \left( \frac{k}{2} \log N_E + \frac{e_k}{a} + \frac{e_k}{a+1}
	+ \big( \tfrac{k+1}{2}\left(\tau + \log\left( \tfrac{1+a}{2}+\tfrac{M+16}{8} \right)\right) + 1 \big)
	+ (k+1)\left( \tfrac1a+\gamma \right) \right),
\end{align*}
where in the last step above, we used Lemma~\ref{lem:bitpar} to estimate the digamma function. Using Lemma~\ref{callmemaybe} to sum the residues from nontrivial zeros over all $k$ with the weights $|\Xi_k|$, we obtain the following bound on the total contribution due to nontrivial zeros:
\begin{align}
	& \le \sum_{k=0}^M |\Xi_k| \sum_\rho \frac{\sqrt{x}}{|\rho+a|^2} \nonumber \\
	&\le \frac{\sqrt x}{2a+1} \left[
		\sum_{k=0}^M k|\Xi_k| \log N_E
		+ (k+1)|\Xi_k| \left(D_1 + \log \tfrac{M+21}{8}\right)
		+ 2|\Xi_k| + 2|\Xi_0|\left(\frac{e_k}{a} + \frac{e_k}{a+1}\right) \right] \nonumber \\
		&\le \frac{ \tfrac{CM}{16} \log(N_E \frac{M+21}{8})  + \frac{CM}{16}D_1
		+ \left( \tfrac2\pi\log M + \pi \right) \left( D_1 + \log \tfrac{M+21}{8} \right)
		+ D_2 + \tfrac{4}{\pi}\log M }{2a+1} \sqrt x,
		\nonumber
\end{align}
where $D_1 \defeq \tau + \frac2a + 2\gamma$, and
$D_2 \defeq 2 \cdot \frac{13}{9}\left(\frac1a + \frac{1}{a+1}\right) + \frac{42}{5}$.
For convenience, we denote the numerator of the fraction in the last expression above by $\Omega$;~i.e. we define
\begin{align*}
	\Omega &\defeq \frac{CM}{16} \log\left(N_E \frac{M+21}{8}\right)  + \frac{CM}{16}D_1+ \\
	&\hphantom{\defeq} \left( \frac2\pi\log M + \pi \right) \left( D_1 + \log \frac{M+21}{8} \right) + D_2 + \frac{4}{\pi}\log M.
\end{align*}
Thus the total contribution from both the trivial and nontrivial zeros is bounded by
\begin{equation}
	\frac{13}{9a^2}
	+ \left( \frac{CM}{16} + \frac{2}{\pi}\log M + \pi \right)\left( \frac{16}{\sqrt x} + \frac{\Upsilon}{\sqrt{x^4-1}}  \right)
	+ \frac{\Omega}{2a+1} \sqrt x
	\label{eq:omegaz}
\end{equation}

\subsubsection{Residue at $s = -a$}
We next estimate the residue due to the pole at $s = -a$. The contribution to this residue from $L(s, \Sym^k E)$ is given by
\[ \left(\log x \cdot \left( -\frac{L'}{L}(-a, \Sym^k E) \right) 
	+ \left( -\frac{L'}{L} \right)' (-a, \Sym^k E)\right)x^{-a}.  \]
The argument presented in this section is analogous to that of Section~\ref{breakfree}; we must simply repeat the proof of Lemma~\ref{lem:yomommastacos} using Lemma~\ref{lem:bitpar} instead of Lemma~\ref{lem:digamma}. The logarithmic derivative of $L(s, \Sym^k E)$ can be expressed in the following way:
\begin{eqnarray}
	-\frac{L'}{L}(s, \Sym^k E) &=& -\frac{L'}{L}(2, \Sym^k E) + \frac{\gamma_k'}{\gamma_k}(s)-\frac{\gamma_k'}{\gamma_k}(2)
	+ \label{eq:dLL2} \\ 
    & & e_{k} \left( \frac{1}{s} + \frac{1}{s-1} - \frac32 \right) + 
  \sum_{\rho} \left( \frac{1}{2-\rho}-\frac{1}{s-\rho} \right).  \nonumber
\end{eqnarray}
Applying Lemma~\ref{lem:bitpar} to estimate the first three terms of~\eqref{eq:dLL2}, we obtain the following bound on the first term of the residue:
\begin{eqnarray*}
	\sum_{k = 0}^M |\Xi_k|\cdot \left\lvert \left( -\frac{L'}{L} \right)(-a, \Sym^k E) \right\rvert & \le & (a+2)\sum_\rho\frac{1}{|\rho+a|^2} + |\Xi_0| \left( \frac1a+\frac1{1+a}+\frac32 \right)+ \\
	& & \sum_{k=0}^M (k+1)|\Xi_k| \left( \zzt + \tau + \log\left( \frac{M+24}{8} \right) \right)+ \\
	& & \sum_{k=0}^M |\Xi_k| \left[ \frac{\Delta}{a} + \frac{28}{5} \right].
\end{eqnarray*}
The first term on the right-hand-side above was already estimated in Section~\ref{johnbarry} and is bounded by $\frac{a+2}{2a+1}\Omega$. Since the bounds we use on $|\Xi_k|$ are decreasing for $k \ge 1$, the above expression can be bounded as follows:
\begin{eqnarray*}
	 \sum_{k = 0}^M |\Xi_k|\cdot \left\lvert \left( -\frac{L'}{L} \right)(-a, \Sym^k E) \right\rvert & \le & \frac{a+2}{2a+1} \Omega + |\Xi_0| \left( \frac1a+\frac1{1+a}+\frac32 + \frac34 \cdot \frac{1}{a} \right)+ \\
	& &  \sum_{k=0}^M (k+1)|\Xi_k| \left( \zzt + \tau + \log\left( \frac{M+24}{8} \right) \right)+ \\
	& & \sum_{k=0}^M |\Xi_k| \left[ \frac{1}{4} \cdot \frac{1}{a} + \frac{28}{5} \right].
\end{eqnarray*}
For convenience, we make the following definitions:
$$D_3 \defeq \frac{7}{4a}+\frac1{1+a}+\frac32,\,\, D_4 \defeq \frac{1}{4a} + \frac{28}{5}, \text{ and }\eta_1 \defeq \zzt + \tau.$$
Now, by repeating the proof of Lemma~\ref{lem:difflogdiff}, we estimate the second term of the residue at $s = -a$ as follows:
\begin{eqnarray*}
	\sum_{k = 0}^M |\Xi_k| \cdot \left\lvert \left( -\frac{L'}{L} \right)'(-a, \Sym^k E) \right\rvert
	& \leq & |\Xi_0| \left( \frac{1}{a^2}+\frac{1}{(1+a)^2} \right) + \frac{\Omega}{2a+1} + \\
	& & \sum_{k=0}^M (k+1)(|\Xi_k|) \cdot \frac{1}{4} \left( \frac{4}{a^2}+\frac{4e_{k}}{(2-a)^2}
	+ \left(\dGam\right)'(1) \right).
 \end{eqnarray*}
Again for convenience, we define 
$$D_5 \defeq \frac{13}{9} \left( a^{-2} + (a+1)^{-2} + (2-a)^{-2} \right) \text{ and } D_6 \defeq a^{-2} + \frac{1}{4}\left( \dGam \right)'(1).$$
Thus, we conclude that the total contribution due to the residue at $s = -a$ (dropping the denominator of $x^a$) is bounded by the following:
\begin{eqnarray*}
	&\le& 	(\log x)\left(
		\frac{a+2}{2a+1}\Omega
		+ \frac{13}{9}D_3
		+ \left( \frac2\pi\log M + \frac{21}{5} \right)D_4 \right) + \\
	& & 	(\log x) \left( \frac{CM}{16}+\frac2\pi\log M + \pi \right)\left(\eta_1 + \log\frac{M+24}{8}\right) + \\
	& & D_5 + \frac{\Omega}{2a+1} + \left( \frac{CM}{16}+\frac2\pi\log M + \pi \right)D_6.
\end{eqnarray*}

\subsection{The Final Estimate}
We now compare the main term with the sum of the residues due to all zeros and the residue at $s = -a$, along with the error due to transitioning to a sum over prime powers.
By collating the main term \eqref{rirrian} with \eqref{eq:omegaz}, \eqref{starbuckslovers}, and the computation in the previous section,
we deduce that it suffices to seek a value of $x$ such that
\begin{eqnarray*}
	\frac{|\Xi_0| x}{(1+a)^2}
	&>& \frac{\Omega}{2a+1} \sqrt x +  \frac{13}{9a^2} + \left( \frac{CM}{16} + \frac2\pi\log M + \pi \right)
		\left( \frac{16}{\sqrt x} + \frac{\Upsilon}{\sqrt{x^4-1}} \right) +\\
	& & 	(\log x)\left(
		\frac{2+a}{2a+1}\Omega
		+ \frac{13}{9}D_3
		+ \left( \frac2\pi\log M + \frac{21}{5} \right)D_4 \right) + \\
	& & (\log x) \left( \frac{CM}{16}+\frac2\pi\log M + \pi \right)\left(\eta_1 + \log\frac{M+24}{8}\right) + \\
	& & D_5 + \frac{1}{2a+1} \Omega + \left( \frac{CM}{16}+\frac2\pi\log M + \pi \right)D_6+ \\
	& & 5\left(2.002\sqrt x + 4x^{1/3} + \log N_E \log x\right).
\end{eqnarray*}
Denote the sum of all terms after the first term in the above equation by $\Phi$.
We now select
\[ M = \left\lceil \frac8\mu \right\rceil \ge 8. \]
Since $|\Xi_0| \ge \mu - \frac{4}{M+1} \ge \mu/2$,
it suffices that our choice of $x$ satisfies
\[ \mu x > 2(1+a^2)\Omega\sqrt x + 2(1+a)^2 \Phi. \]
We then make the following choices:
\[
	a = \frac14 \sqrt[3]{\frac{\log(319/8)}{\Theta}},
	\quad\text{where}\quad
	\Theta = \log \left( N_E \cdot \frac{M+21}{8} \right).
\]
We can write $a$ as $a = B_0 \Theta^{-1/3}$, where $B_0 = \frac14\sqrt[3]{\log(319/8)}$. 
(We have $a \le \frac14$ simply because $N_E \ge 11$, $M \ge 8$.)
If our choice of $x$ satisfies $\mu x > 2(1+a^2)\Omega\sqrt x + 2(1+a)^2 \Phi$,
then $x$ also necessarily satisfies 
\[ \sqrt x 
	> \frac{2(1+a)^2 \Omega }{\mu}
	> \frac{2(1+a)^2 \left( \frac{CM}{16} \Theta \right)}{\mu}
	\ge \frac{(1+a^2)C\Theta}{\mu^2}. \]
A computer calculation implies that the above inequality yields
$2(1+a)^2\Phi < B_1\sqrt x$, where $B_1 \defeq 276.572$.
Thus, in order to show that $W > 0$, it suffices to choose $x$ so that
\[ \mu x > \left( 2(1+a^2)\Omega + B_1 \right)\sqrt x. \]
If we expand $2(1+a^2)\Omega$, we find that
\begin{eqnarray*}
	2(1+a^2)\Omega
	&=& \frac{CM}{8} \left( \Theta + \Theta a^2 + \frac2a \right)
	+ \frac{CM}{8}\left( 2a + (1+a^2)(\tau+2\gamma) \right) +\\
	& & 2(1+a^2)\left( 
	\left( \frac2\pi\log M + \pi \right) \left( D_1 + \log \frac{M+21}{8} \right)
	+ D_2 + \frac{4}{\pi}\log M \right).
\end{eqnarray*}
The terms on the second line above are evidently $\ll \Theta^{1/3} (\log M)^{2/3}$
and in fact, a computer calculation shows that
\[ 2(1+a^2)\Omega < \frac{CM}{8} \left( \Theta + \Theta a^2 + \frac2a + B_2 \right)
	+ B_3 \Theta^{1/3} (\log M)^{2/3}  \]
where $B_2 \defeq 2 \cdot \frac14 + (1+(\frac14)^2)(\tau+2\gamma) = 6.225$ and $B_3 \defeq 105.007$.
Recalling our choice of $M$, we see that it suffices to choose $x$ satisfying
\[
	\mu \sqrt x >
	\frac{C}{\mu} \left( \Theta + \Theta a^2 + \frac{2}{a} + B_2 \right)
	+ B_3 \Theta^{1/3} (\log M)^{2/3} + B_1.
\]
Since we took $a = B_0 \Theta^{-1/3}$, letting $B_4 \defeq C(B_0^2 + 2/B_0) = 111.106$
yields that
\[
	\mu \sqrt x >
	\frac{1}{\mu} \left( C\Theta + B_4\Theta^{1/3} + B_2  \right)
	+ B_3 \Theta^{1/3} (\log M)^{2/3} + B_1.
\]
Thus, it suffices to have $x$ satisfy
\[ x > \left( \frac{C\Theta + \left(B_4 + B_3\mu(\log 8/\mu)^{2/3}\right)\Theta^{1/3} + B_2 + B_1\mu}{\mu^2}  \right)^2. \]
If we let $B_5 \defeq B_4 + B_3(\log 8)^{2/3} + B_2/(\log 319/8)^{1/3} = 286.606$, 
we obtain Theorem~\ref{thm:roar} in the completely explicit form
\begin{equation}
	p < \left( \frac{ C\Theta + 287\sqrt[3]{\Theta} + 277\mu}{\mu^2} \right)^2.
	\label{eq:explicit_roar}
\end{equation}
The computation of the various constants used in the above proof of Theorem~\ref{thm:roar} can be found in the ancillary program file (see Appendix~\ref{tabwars} regarding this file).

\appendix

\section{Notes on the Optimality of the Estimates}\label{rawrconstants}
In this section, we discuss how the constants that appear in the explicit statements (see~\eqref{eq:explicit_main} and~\eqref{eq:explicit_roar}) of Theorems~\ref{thm:main} and~\ref{thm:roar} arise; in particular, we explain why these constants cannot be easily improved without utilizing different methods.

Given a Galois extension $L/K$ of a number field $K$, it is shown in~\cite{bach} that the least norm of a prime $\pp \in  \OO_K$ is bounded by $\left( (1+o(1)) \log d_L \right)^2$, where $d_L$ is the absolute discriminant of $L$. Given that our techniques are modeled after those employed in~\cite{bach}, it is reasonable to expect a bound of $\left( (1+o(1)) \log N \right)^2$ for Theorem~\ref{thm:main}, rather
than $\left( (32+o(1))\log N \right)^2$.
The coefficient of ``$32$'' in the main term of Theorem~\ref{thm:main} arises from estimating the contribution due to the sum~\eqref{eq:blame} of residues from nontrivial zeros of $L(s, \PSY{i}{j})$; specifically, the ``$32$'' appears upon computing the following sum:
\begin{equation}
	\frac{1}{|c_0|^2}\sum_{0 \leq i,j \leq 2} |c_{ij}| \cdot i(j+1) = 32.
	\label{eq:opt32}
\end{equation}
Here, the coefficients $c_{ij}$ are defined in Section~\ref{coldasice} by $c_{ij} = c_ic_j$, where $f_k(t) \defeq c_0U_0(t) + (-1)^{k+1}c_1U_1(t) + c_2U_2(t)$ is the minorant that we are using in order to detect sign chances in $\cos \theta_{k,p}$ for each $k \in \{1,2\}$. It is clear that the coefficient of the main term will decrease if $c_0$ is large relative to the numerator of~\eqref{eq:opt32}, but for degree-$2$ polynomials, the choice $(c_0,c_1,c_2) = (1,2,1)$ is optimal in this sense. Moreover, our choice of $f_1,f_2$ is optimal in the sense that the constant terms of $f_1, f_2$, when written in the basis of Chebyshev polynomials of the second kind, are as close to the Sato-Tate measures of the indicator functions of the intervals $[0, \pi/2]$ and $[\pi/2, \pi]$. 
Perhaps counter-intuitively, using a higher-degree polynomial appears to make the coefficient of the main term worse,
since the number of coefficients $c_i$ increases, causing the numerator of~\eqref{eq:opt32} to grow large.

Similarly, the leading term in Theorem~\ref{thm:roar} arises when we compare the value of $|\Xi_0|$ with the estimate
\[ \sum_{k=0}^M k|\Xi_k| = \left( \frac23+\frac2\pi + o(1) \right)M \]
In this case, to ensure that we have $\Xi_0 \ge 0$, we need $M \gg \frac{1}{\mu}$; here, the choice $M = \lfloor 8/\mu \rfloor$ is optimal, and this choice yields that
$|\Xi_0| \ge \half\mu$. It follows that the coefficient of the main term is given by $C = 32(1/3 + 1/\pi)$. In summary, the Fourier-analytic techniques that we use (specifically, the Beurling-Selberg minorant $S$ of the indicator function of an interval)
are entirely responsible for the surprisingly large leading coefficient. In contrast to the situation of Theorem~\ref{thm:main},
where the coefficients of the polynomials $f_1, f_2$ are known explicitly, the situation is worsened in Theorem~\ref{thm:roar}: since the Fourier coefficients of the Beurling-Selberg minorant cannot be expressed in an easy-to-use form, we must estimate them instead of using explicit formulas.

The constants in the error terms as presented in~\eqref{eq:explicit_main} and~\eqref{eq:explicit_roar}, all of which are on the order of $100$, ultimately arise in the same way from our choice of minorant. For example, in Theorem~\ref{thm:main}, the term of order $\sqrt[3]{\log N}$ appears
when we select $a \asymp (\log N)^{-1/3}$ in the largest error term on the right-hand-side of~\eqref{vinogradova}, which is given by
\[ 32a^2 \log N + \frac{130}{a}. \]
The large coefficient of $130$, which arises from the factor of $128$ that appears in~\eqref{eq:blame}, is what causes the constant $C_4$ to be as large as it is. Even the transition to summing over prime powers, which is normally a benign step,
introduces a sizable contribution to the error term, since we are compelled to use the bound $|c_0|^{-2} \cdot |f_1(t)| \cdot |f_2(t)| \le 64$,
which leads to the large error term $64 \cdot 2.002 \sqrt x$ (on the right-hand-side of~\eqref{vinogradova}, this term is combined with other terms to yield the error term $C_2\sqrt{x}$).

\section{Tabulation of Variables and Constants}\label{tabwars}
For convenience, we now list the definitions of
all variables and constants used throughout the rest of this paper.
Each variable listed below is also defined when introduced in the body of the paper. Note that $a$ is a positive real number in the interval $(0, 1/4]$;
we will select it when appropriate, at the end of the proofs of Theorems~\ref{thm:main} and~\ref{thm:roar}. Computations of some of the constants listed below and other calculations that are essential to the paper are included in an ancillary program file, which will be provided to the reader upon request.

\subsection{Notation in the Proof of Theorem~\ref{thm:main}}
For the proof of Theorem~\ref{thm:main}, we define the following constants:
\begin{align*}
	\tau &= \frac{11}{3} + \log\pi \approx 4.811 \\
	\eta &= \log(12) + 2\tau + \zzt \approx 12.678 \\
	\Upsilon &= 0.337 \\
	\Psi &= 426.875 \\
	\intertext{The following variables depend on $a$:}
	A_1 &= \half(\tau + \log(a/2+7/2)) + \frac1a + \gamma \\
	A_2 &= \frac{2}{2a+1}\left( 64A_1 + \frac{2a+1}{a(a+1)} \right) \\
	A_3 &= 64\cdot2.002 + A_2 \\
	A_4 &= \frac{\Psi}{4} + \frac{2}{a^2} + \frac{1}{(1+a)^2} + A_2 \\
	A_5 &= \frac{1}{a + 1} + \frac{1}{a} + \frac{3}{2} \\
	A_6 &= (a+2)A_1 + 64\eta + 96 + \frac1a + A_5
	\intertext{The following constants are absolute:}
	C_0 &= \frac{\sqrt[3]{\log 121}}{4} \approx 0.422 \\
	C_1 &= \half(\tau+\log(29/8)) + \gamma \approx 3.627 \\
	C_2 &= 64\cdot2.002\cdot\frac{25}{16} + 2 + 128\cdot\frac{25}{24}C_1 + 64 \approx 749.779 \\
	C_3 &\approx 56.958 \\
	C_4 &= 32C_0^2 + \frac{130}{C_0} \approx 314.042 \\
\end{align*}
\subsection{Notation in the Proof of Theorem~\ref{thm:roar}}
We retain the constants $\tau$, $\eta$, and $\Upsilon$.
We define the absolute constant
\begin{align*}
	C &= \frac{32}{3} + \frac{32}{\pi} \approx 20.853 \\
	\intertext{The following variables depend on $a$:}
	D_1 &= \tau + \frac2a + 2\gamma \\
	D_2 &= \frac2a + \frac{2}{a+1} + \frac{21}{5} \\
	D_3 &= \frac{7}{4a} + \frac{1}{1+a} + \frac32 \\
	D_4 &= \frac{1}{4a} + \frac{28}{5} \\
	D_5 &= \frac{13}{9} \left( a^{-2} + (1+a)^{-2} \right) \\
	D_6 &= \half \left( 4a^{-2} + 4(2-a)^{-2} + \left( \dGam \right)'(1) \right).
	\intertext{The following constants are absolute:}
	B_0 &= \frac{\sqrt[3]{\log(319/8)}}{4} \approx 0.386 \\
	B_1 &= 276.572 \\
	B_2 &= \frac{17}{16}(\tau+2\gamma) + \half \approx 6.839 \\
	B_3 &= 105.007 \\
	B_4 &= C(B_0^2+\tfrac{2}{B_0}) \approx 111.106 \\
	B_5 &= (\log 8)^{\frac23}B_3 + B_4 + \frac{B_2}{\log(330/8)^{1/3}} \approx 286.606
\end{align*}


\section*{Acknowledgments}
\noindent This research was supervised by Ken Ono at the Emory University Mathematics REU and was supported by the National Science Foundation (grant number DMS-1250467). We would like to thank Ken Ono and Jesse Thorner for offering their advice and guidance and for providing many helpful conversations and valuable suggestions on the paper. We would also like to acknowledge Hugh L.~Montgomery for helpful discussions on Vaaler's Lemma. We would finally like to thank the anonymous referee for providing beneficial comments on earlier versions of the paper. The authors used SageMath 6.7 and \emph{Mathematica} 10.0 for explicit calculations.

\normalsize
\bibliographystyle{amsxport}
\bibliography{biblio2}

\end{document}